\documentclass[11pt,reqno]{amsart}
\usepackage{amssymb,amsmath}
\usepackage{amsthm}
\usepackage{color,graphicx}
\usepackage{hyperref}
\usepackage{mathabx}

\newcommand{\sech}{{\rm \,sech}}
\newcommand{\R}{{\mathbb R}}
\newcommand{\Lum}{{\mathcal{L}_{Re}}}
\newcommand{\Ldois}{{\mathcal{L}_{Im}}}

\newcommand{\ve}{{\varepsilon}}
\newcommand{\Ll}{\mathbb{L}}
\newcommand{\Hh}{\mathbb{H}}

\numberwithin{equation}{section}

\newtheorem{theorem}{Theorem}[section]
\newtheorem{proposition}[theorem]{Proposition}
\newtheorem{remark}[theorem]{Remark}

\newtheorem{definition}[theorem]{Definition}

\begin{document}
\vglue-1cm \hskip1cm
\title[Standing waves for the Schr\"odinger-Kirchhoff equation]{Existence and orbital stability of standing waves for the 1D Schr\"odinger-Kirchhoff equation.}


\maketitle

\begin{center}
{\bf F\'abio Natali}

{\scriptsize    Department of Mathematics, State University of
	Maring\'a, Maring\'a, PR, Brazil. }\\
{\scriptsize fmanatali@uem.br }

\vspace{3mm}

{\bf Eleomar Cardoso Jr.}

{\scriptsize   Federal University of Santa Catarina, Blumenau, SC, Brazil.}\\
{\scriptsize eleomar.junior@ufsc.br}
\end{center}

\begin{abstract}
In this paper we establish the orbital stability of standing wave solutions associated to the one-dimensional Schr\"odinger-Kirchhoff equation. The presence of a mixed term gives us more dispersion, and consequently, a different scenario for the stability of solitary waves in contrast with the corresponding nonlinear Schr\"odinger equation. For periodic waves, we exhibit two explicit solutions and prove the orbital stability in the energy space.

\end{abstract}

\section{Introduction}

This paper addresses orbital stability results for the one-dimensional evolutionary Schr\"odinger-Kirchhoff equation
\begin{equation}\label{nls}
iu_t+\left(1+\int_{\mathbb{B}}|u_x|^2dx\right)u_{xx}+|u|^{2r}u=0,
\end{equation}
where $r\geq1$ and $u=u(x,t)$ is a complex-valued function defined over the set $\mathbb{B}\times \R^{+}$. Here, $\mathbb{B}=\mathbb{R}$ or $\mathbb{B}=\mathbb{T}$ and in the second case, we restrict ourselves to periodic solutions with period $2\pi$ by convenience. Equation $(\ref{nls})$ arises in
quantum mechanics and describes the dynamics of the particle in a non-relativistic setting (see \cite{anto}). The additional term $\int_{\mathbb{B}}|u_x|^2dx$ represents a magnetic potential.\\
\indent From the mathematical point of view, the mixed dispersive term in $(\ref{nls})$ provides us a lack of scaling invariance which is a very important property to determine, for instance, a threshold value for obtaining global well-posedness results in the energy space.\\
\indent  Another important mathematical aspect concerning equation $(\ref{nls})$ is the existence of standing waves. They are finite-energy waveguide solutions of \eqref{nls} having the form
\begin{equation}\label{standing}
u(x,t)=e^{i\omega t}\phi(x),
\end{equation}
where $\omega$ is a positive constant representing the frequency of the wave and $\phi:\mathbb{B}\to\R$ is a smooth function  satisfying: $\phi^{(n)}(x)\rightarrow 0$ for all $n\in\mathbb{N}$, as $|x|\to+\infty$, when $\mathbb{B}=\R$ or  $\phi^{(n)}(0)=\phi^{(n)}(2\pi)$ for all $n\in\mathbb{N}$, when $\mathbb{B}=\mathbb{T}$.

Here, we consider smooth curves of standing waves $\phi$ depending on $\omega$. To do so, we substitute \eqref{standing} into \eqref{nls} to obtain the following nonlinear ODE
\begin{equation}\label{soleq}
-\left(1+\int_{\mathbb{B}}\phi'^2dx\right)\phi''+\omega \phi-\phi^{2r+1}=0.
\end{equation}
\indent First, we consider the case $\mathbb{B}=\mathbb{R}$. Solitary waves with hyperbolic secant profile of the form
\begin{equation}\label{solitary}
\phi(x)=a\sech^{\frac{1}{r}}(bx),
\end{equation} satisfies equation $(\ref{soleq})$. Here, $a$ and $b$ are smooth parameters depending on $\omega$ to be determined later on.\\
\indent When $\mathbb{B}=\mathbb{T}$, positive and periodic standing waves associated to the equation $(\ref{soleq})$ for some particular cases of $r$ are well known. Indeed, if one considers the case $r=1$, we have
\begin{equation}\label{per1}
\phi(x)=a{\rm dn}(bx,k).
\end{equation}
Now, for the case $r=2$, a positive and periodic wave can be expressed by
\begin{equation}\label{per2}
\phi(x)=\frac{a{\rm dn}(bx,k)}{\sqrt{1-\alpha {\rm sn}^2(bx,k)}},
\end{equation}
where ${\rm dn}$ and ${\rm sn}$ are the Jacobi Elliptic Functions called \textit{dnoidal} and $\textit{snoidal}$, respectively. The value $k\in(0,1)$ is called modulus of the elliptic function and parameters $a$, $b$, $\alpha$, and $\omega$ depend smoothly on the modulus $k$.

\indent Our main interest in this paper is to show that the standing
wave $\phi$ with profiles  \eqref{solitary}, $(\ref{per1})$ or $(\ref{per2})$ are
orbitally stable in the complex energy space $\Hh^1$ (see the definition of $\Hh^n$, $n\in\mathbb{N}$ ahead). As far we can see, the standing wave $\phi$
is \textit{orbitally stable}, if the profile of an initial data
$u_0$ for (\ref{nls}) is close to $\phi$, then the associated
evolution in time $u(t)$, with $u(0)=u_0$, remains close to $\phi$,  up to symmetries, for
all values of $t\geq 0$ (see Definition \ref{stadef} for the precise definition).\\
\indent The strategy to prove the orbital stability/instability of standing waves is based on the developments contained in \cite{gss1}, \cite{gss2} and \cite{np} where the authors have been established sufficient conditions for obtaining the orbital stability/instability of standing waves for abstract Hamiltonian systems of the form
\begin{equation}\label{stha}
u_t(t)=JE'(u(t))
\end{equation}
posed on a Hilbert space $X$, where $J$ is an invertible bounded operator in $X$ and $E'$ represents the Fr\'echet derivative of a energy functional $E$. In our context, $J=\left(\begin{array}{cccc}
0 & -1\\
1 & 0
\end{array}\right)$ and $E$ is given by
\begin{equation}\label{E}
E(u)=\frac{1}{2}\int_{\mathbb{B}}|u_x|^2dx+\frac{1}{2}\left(\int_{\mathbb{B}}|u_x|^2dx\right)^2-\frac{1}{2r+2}\int_{\mathbb{B}}|u|^{2r+2}dx.
\end{equation}

One of the most crucial
assumptions in such theories are that the underlying standing wave
belongs to a smooth curve of standing waves, $\omega\in
I\mapsto\phi_{\omega}$, depending on the phase
parameter $\omega$. The existence of a smooth curve is very useful to calculate the precise sign of
\begin{equation}\label{signd}\frac{d}{d\omega}\int_{\mathbb{B}}\phi^2dx,
\end{equation}
which plays an important role in the orbital stability/instability analysis. \\
\indent Another important fact concerns some results of spectral analysis for the linear operator arising in the linearization of the equation around the standing wave. In our context, such an operator turns out to be a matrix operator containing Schr\"odinger-type operators in the principal diagonal. In our paper, we need to use some tools contained in the classical Sturm-Liouville ($\mathbb{B}=\R$) and Floquet ($\mathbb{B}=\mathbb{T}$) theories.\\
\indent We can summarize our stability results by stating the following Theorem:
\begin{theorem}\label{maintheo}
(i) For the case $\mathbb{B}=\mathbb{R}$, the smooth curve of standing wave solutions associated to the equation $(\ref{soleq})$ is orbitally stable in $\Hh^1$ provided that $r=1,2$. If $r=4$ the solitary wave $\phi$ is orbitally unstable in $\Hh_e^1$.\\
(ii) For the case  $\mathbb{B}=\mathbb{T}$, the smooth curve of standing wave solutions associated to the equation $(\ref{soleq})$ is orbitally stable in $\Hh^1$ if $r=1,2$.
\end{theorem}

\indent At least one contribution concerning the orbital stability of ground states associated with the equation $(\ref{nls})$ can be found in the current literature. In the three-dimensional case and $r\in (0,2/3)$, the authors in \cite{zhang} showed that the solitary waves which minimize $E$ in $(\ref{E})$ subject to fixed mass are orbitally stable in $\Hh^1$. The proof relies tools contained in the abstract theory in \cite{we} and a suitable result of uniqueness for the associated solitary waves.\\
\indent Next we perform a comparison between our results and the orbital stability of solitary standing waves associated with the well known nonlinear Schr\"odinger equation
\begin{equation}\label{regnls}
iu_t+u_{xx}+|u|^{2r}u=0.
\end{equation}
We can find a considerable number of contributors concerning this topic (see for instance \cite{gss1}, \cite{gss2} and related works). In fact, first of all it is well known that a hyperbolic secant profile of the form $(\ref{solitary})$ and depending smoothly on $\omega>0$ also solves equation $(\ref{soleq})$. This fact and the presence of scaling invariance in equation $(\ref{regnls})$ give us
\begin{equation}\label{derL2}
\frac{d}{d\omega}\int_{\mathbb{R}}\phi^2dx=c_r(2-r)\int_{\mathbb{R}}\phi^2dx,
\end{equation}
where $c_r$ is a positive constant depending on $r$. The solitary wave $\phi$ is orbitally stable for $1<r<2$ and unstable in $\Hh_{e}^1$ for $r>2$. In the degenerate case $r=2$ the methods aforementioned can not be applied since $\phi$ is a saddle point of the problem in finding a minimum/maximum point of the associated energy with fixed momentum. In our case, the absence of the scaling invariance does not determine a good equality as in $(\ref{derL2})$ and the choice of the conditional exponent to get the global well-posedness in $\Hh^1$ (see Subsection 2.3), namely $r=4$, give us the orbital instability of the corresponding solitary wave in $\Hh_e^1$.\\
\indent Regarding the orbital stability of periodic waves, we have two important references. For the case $r=1$, it is well known that periodic waves with dnoidal type as in $(\ref{per1})$ are solutions of the equation $(\ref{regnls})$ and the orbital stability of these waves have been determined in \cite{angulo1}. The critical case $r=2$ was treated in \cite{AN2}, where the authors established orbital stability in $\Hh^1$ and orbital instability in $\Hh_e^1$ for the positive periodic wave in $(\ref{per2})$. Important to mention that in both cases, the authors have been used a combination of the main arguments in \cite{gss1}, \cite{gss2} and \cite{we} for the stability and \cite{gss1} for the instability. Our results give us that for both cases $r=1$ and $r=2$ the corresponding dnoidal waves are orbitally stable.

The paper is organized as follows. In Section \ref{not} we introduce some notations and additional results. Section 3 is devoted to the spectral theory associated to the linearized operator around the wave $(\phi,0)$. In Section 4 we show our stability result for the standing wave \eqref{standing}. 

\section{Notation and Preliminaries} \label{not}

\subsection{Notations.} We present some notation used throughout the paper. Given $n\in\mathbb{N}$, by $H^n:=H^n(\mathbb{B})$ we denote the usual Sobolev space of real-valued functions. In particular $H^0(\mathbb{B})\simeq L^2(\mathbb{B})$. The scalar product in $H^n$ will be denoted by $(\cdot,\cdot)_{H^n}$ and the norm induced by this inner product is indicated by $||\cdot||_{H^n}$. We set
$$
\Ll^2=L^2(\mathbb{B})\times L^2(\mathbb{B})\qquad
\mbox{and}\qquad
\Hh^n=H^n(\mathbb{B})\times H^n(\mathbb{B}).
$$
Such spaces are endowed with their usual norms and scalar products.  $\mathbb{H}_{e}^1$ indicates the space $\Hh^1$ constituted by even functions.

Besides $E$ in $(\ref{E})$, equation \eqref{nls} conserves the mass
\begin{equation}\label{F}
F(u)=\frac{1}{2}\int_{\mathbb{B}} |u|^2\,dx.
\end{equation}

Equation \eqref{nls} can also be viewed   as a real Hamiltonian system. In fact,
by writing $u=P+iQ$ and separating real and imaginary parts, we see that
\eqref{nls} is equivalent to the system
\begin{equation}\label{equisys}
\left\{\begin{array}{cccc}
\displaystyle P_t+\left(1+\int_{\mathbb{B}}P_x^2+Q_x^2dx\right)Q_{xx}+Q(P^2+Q^2)^{r}=0,\\
\displaystyle-Q_t+\left(1+\int_{\mathbb{B}}P_x^2+Q_x^2dx\right)P_{xx}+P(P^2+Q^2)^r=0.
\end{array}\right.
\end{equation}
Moreover, the quantities \eqref{E} and \eqref{F} become
\begin{equation}\label{hamiltocons1}
E(P,Q)=\frac{1}{2}\int_{\mathbb{B}} P_x^2+Q_x^2dx+\frac{1}{2}\bigg(\int_{\mathbb{B}}P_x^2+Q_x^2 dx\bigg)^2-\frac{1}{2r+2}\int_{\mathbb{B}}(P^2+Q^2)^{r+1}dx,
\end{equation}
\begin{equation}\label{mass1}
F(P,Q)=\frac{1}{2}\int_{\mathbb{B}} (P^2+Q^2)\,dx.
\end{equation}
Consequently, \eqref{equisys} or, equivalently, \eqref{nls} can be written as
\begin{equation}\label{hamiltonian}
\frac{d}{dt}U(t)=J E'(U(t)), \qquad U={P\choose Q},
\end{equation}
where $J$ is given by
$
J=\left(\begin{array}{cccc}
 0 & -1\\
1 & 0
\end{array}\right)
$
and it is easy to see that
$
J^{-1}=-J.
$

Equation \eqref{nls} is invariant under the unitary action of
rotation and translation, that is, if $u=u(x,t)$ is a solution of
\eqref{nls} so are $e^{-i\theta}u$ and $u(x-s,t)$, for any real numbers $\theta$ and $s$. Equivalently,
this means if $U=(P,Q)$ is a solution of \eqref{hamiltonian}, so are
\begin{equation}\label{l0}
T_1(\theta)U:=\left(\begin{array}{cccc}
 \cos\theta & \sin\theta\\
-\sin\theta & \cos\theta
\end{array}\right)
\left(\begin{array}{c}
 P\\
 Q
\end{array}\right)
\end{equation}
and
\begin{equation}\label{l00}
T_2(s)U:= \left(\begin{array}{c}
 P(\cdot-s,\cdot)\\
  Q(\cdot-s,\cdot)
\end{array}\right).
\end{equation}

The actions $T_1$ and $T_2$ define unitary groups in $\Hh^2$ with
infinitesimal generators given, respectively, by
\begin{equation*}
T_1'(0)U:=\left(\begin{array}{cccc}
 0 & 1\\
-1 & 0
\end{array}\right)
\left(\begin{array}{c}
 P\\
 Q
\end{array}\right)\equiv -J\left(\begin{array}{c}
 P\\
 Q
\end{array}\right)
\end{equation*}
and
\begin{equation*}
T_2'(0)U:=\partial_x \left(\begin{array}{c}
 P\\
  Q
\end{array}\right).
\end{equation*}

In this context, a standing wave solution having the form \eqref{standing} becomes a solution of \eqref{equisys} of the form
\begin{equation}\label{mstanding}
U(x,t)={\phi(x)\cos(\omega t)\choose \phi(x)\sin(\omega t)}.
\end{equation}
Thus, the  function $U$ in \eqref{mstanding},  with $\phi$ given by one of the solutions in $(\ref{solitary})$, $(\ref{per1})$ or $(\ref{per2})$ is a standing wave solution of \eqref{equisys}.

Next, from \eqref{soleq} we obtain that $(\phi,0)$ is a critical point of
the functional $E+\omega F$ in the sense that
\begin{equation}\label{crit}
E'(\phi,0)+\omega F'(\phi,0)=0.
\end{equation}

To simplify the notation, we set
\begin{equation}\label{Phi}
\Phi:=(\phi,0)
\end{equation}
and let us define
\begin{equation}\label{linear}
G:=E+\omega F.
\end{equation}
We see from $(\ref{crit})$ that $G'(\Phi)=0$.

The splitting $u=P+iQ$ enables us to introduce the diagonal linear operator given by
\begin{equation}     \label{L}
\mathcal{L}:= \left(\begin{array}{cccc}
\mathcal{L}_{Re} & 0\\\\
0 & \mathcal{L}_{Im}
\end{array}\right),
\end{equation}
where
\begin{equation}\label{L1}
\mathcal{L}_{Re}:= -\left(1+\int_{\mathbb{B}}\phi'^2dx\right)\partial_x^2-2(\phi',\partial_x\cdot)_{L^2}\phi''+\omega-(2r+1)\phi^{2r}
\end{equation}
and
\begin{equation}\label{L2}
\mathcal{L}_{Im}:= -\left(1+\int_{\mathbb{B}}\phi'^2dx\right)\partial_x^2+\omega-\phi^{2r}.
\end{equation}
This operator appears in the linearization of \eqref{equisys} around the wave $\Phi=(\phi,0)$. The knowledge of its spectrum is cornerstone in the analysis contained in Section 3.

\subsection{Existence of solitary and periodic standing waves}\label{reviewsec}
In this subsection, we show the existence of explicit standing wave solutions associated to the equation $(\ref{soleq})$. First, we prove the existence of solitary waves as in $(\ref{solitary})$.\\

\noindent {\bf Case 1.} $\mathbb{B}=\R$.\\  We prove that particular cases of $r$ the profile in $(\ref{solitary})$ is a solution for the equation $(\ref{soleq})$ and we establish that the values of $a$ and $b$ are given in terms of $\omega\in I$. Here, $I$ is an unbounded open interval contained in $\R$. After that, we deal with the general case by supposing additional conditions to get $a$ and $b$ depending on $\omega$ in open bounded intervals.\\

\noindent $\bullet$ For instance, if $r=1$ one has that

\begin{equation}\label{ar1}
a=\frac{1}{2b^2}\left(6b(\omega-b^2)\right)^{\frac{1}{2}}
\end{equation}
and $b$ is a complicated function in terms of $\omega$ given by

\begin{equation}\label{br1}
b=\frac{1}{4\omega}\left(\left(24\omega^3-1+4\sqrt{3}\sqrt{\frac{12\omega^3-1}{\omega}}
\omega^2\right)^{\frac{1}{3}}+\frac{1}{\left(24\omega^3-1+4\sqrt{3}\sqrt{\frac{12\omega^3-1}{\omega}}\omega^2\right)^{\frac{1}{3}}}-1\right).
\end{equation}
\indent We see from $(\ref{br1})$ that $\omega>\frac{1}{\sqrt[3]{12}}\approx0.43679$ and we can define $b$ and $a$ in terms of $\omega$ with $0.57235\approx\frac{1}{4}\sqrt[3]{12}<b<b_0\approx 0.9$ and $a>\frac{1}{24}\sqrt{96}\sqrt[3]{12}\approx 0.93467$. Moreover, $a$ and $b$ are strictly increasing functions in terms of $\omega$ according to the following pictures:
\newpage
\begin{figure}[!htb]
	\centering
	\begin{minipage}[b]{0.4\linewidth}
		\includegraphics[scale=0.35]{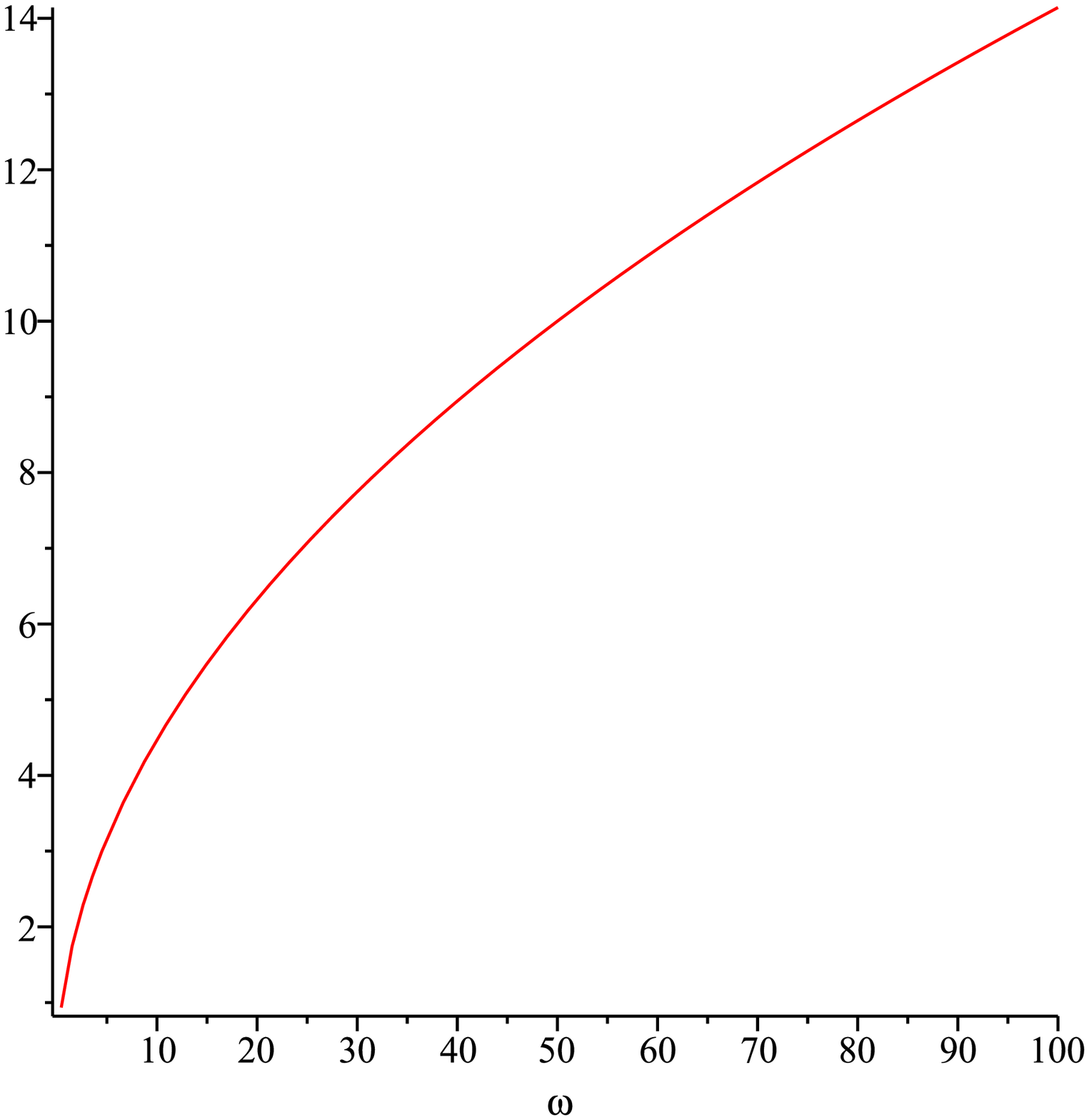}
		\caption{Graphic of $a$.}
	\end{minipage}\hfill
	\begin{minipage}[b]{0.4\linewidth}
		\includegraphics[scale=0.35]{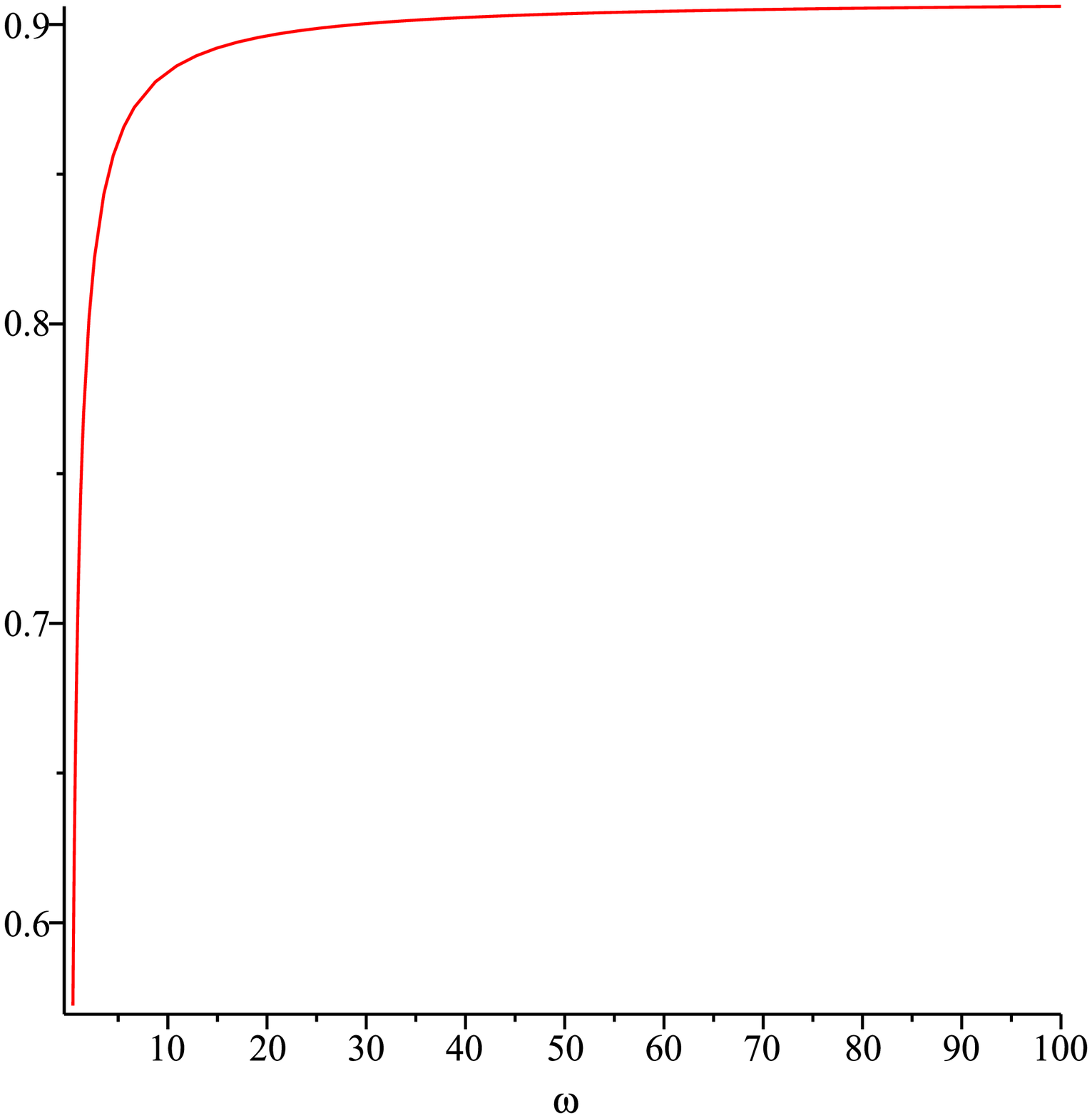}
		\caption{Graphic of $b$.}
	\end{minipage}
\end{figure}

\noindent $\bullet$ Let us consider the case $r=2$. For $\omega>\frac{8}{9\pi}\approx 0.28294$, we see that $a$ and $b$ can also be explicitly determined but we omit their expressions to simplify the notations. As above, $a$ and $b$ are strictly increasing functions in terms of $\omega$ according to the following pictures:

\begin{figure}[!htb]
	\centering
	\begin{minipage}[b]{0.4\linewidth}
		\includegraphics[scale=0.35]{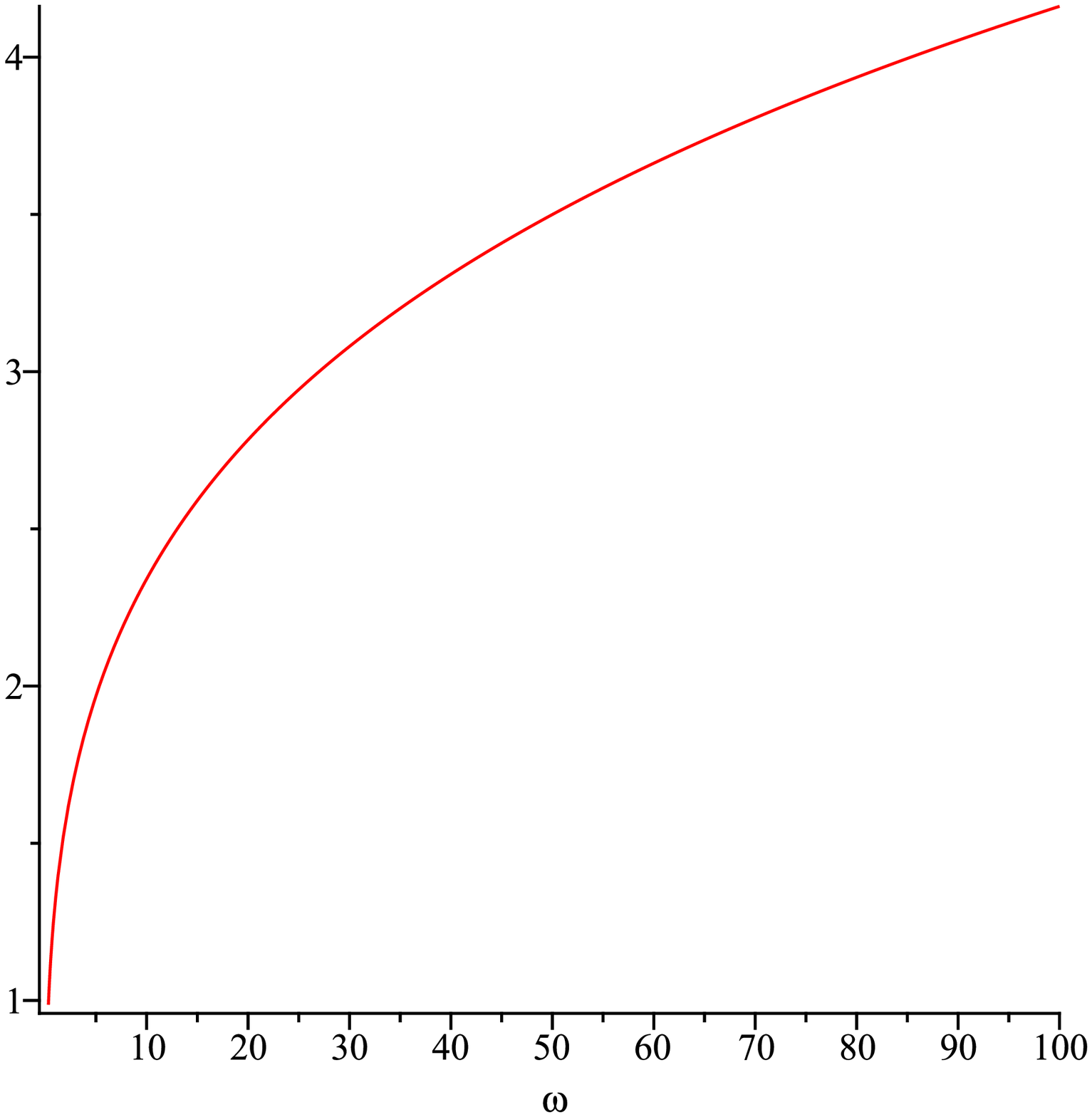}
		\caption{Graphic of $a$.}
	\end{minipage}\hfill
	\begin{minipage}[b]{0.4\linewidth}
		\includegraphics[scale=0.35]{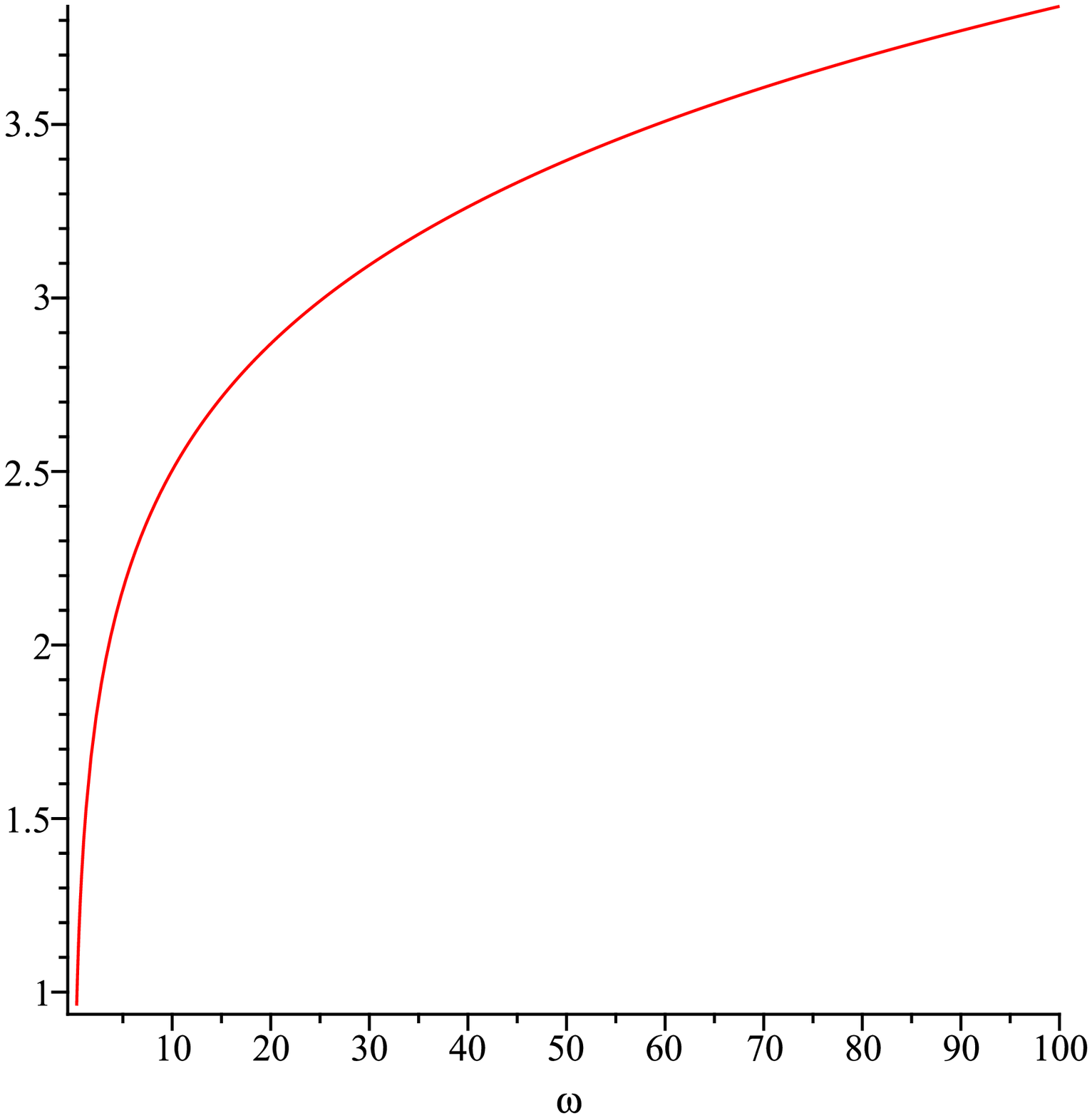}
		\caption{Graphic of $b$.}
	\end{minipage}
\end{figure}

\noindent $\bullet$ If $r=4$ we have $a$ and $b$ can be explicitly determined in terms of $\omega>\omega_0\approx0.19594$. Next pictures show us the behavior of $a$ and $b$ in terms of $\omega$.
\newpage
\begin{figure}[!htb]
	\centering
	\begin{minipage}[b]{0.4\linewidth}
		\includegraphics[scale=0.35]{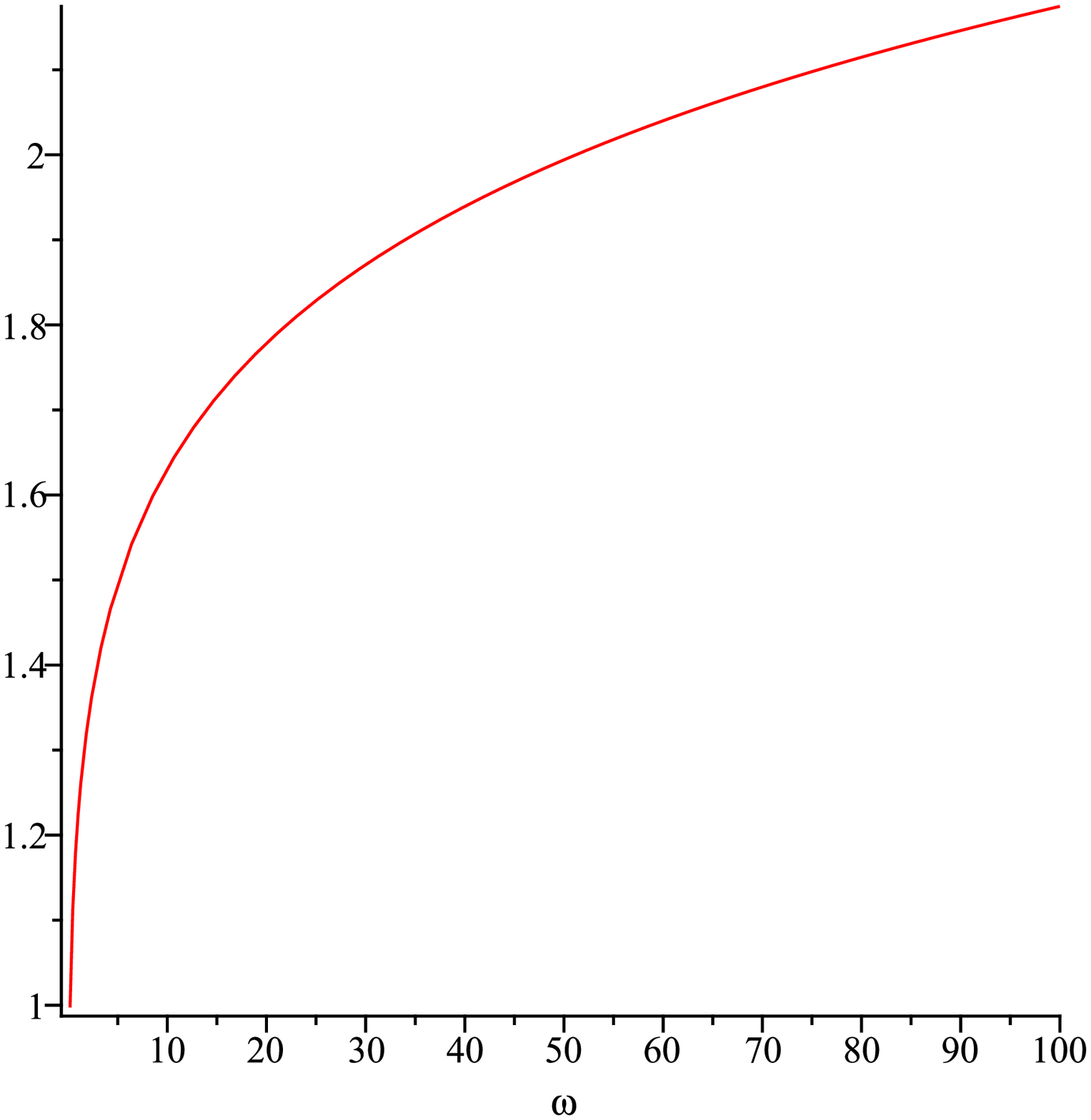}
		\caption{Graphic of $a$.}
	\end{minipage}\hfill
	\begin{minipage}[b]{0.4\linewidth}
		\includegraphics[scale=0.35]{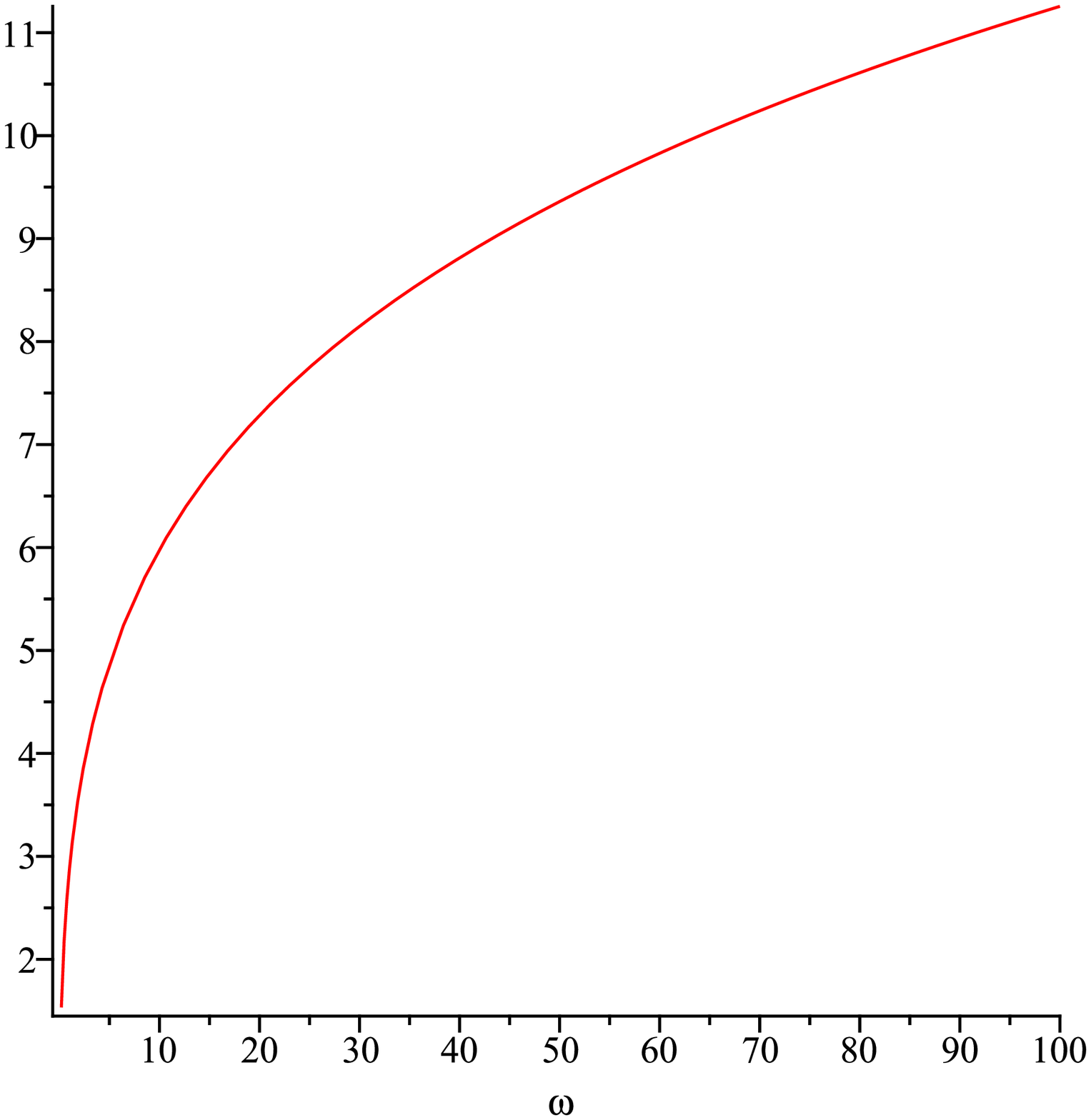}
		\caption{Graphic of $b$.}
	\end{minipage}
\end{figure}

\noindent $\bullet$ The general case can be determined under a set of local conditions. First of all, let us substitute the general form in $(\ref{solitary})$ into $\int_{\R}\phi'^2dx$ to get

\begin{equation}\label{gensolit1}
\int_{\R}\phi'^2dx=\frac{a^2b}{r^2}\int_{\R}{\rm sech^{\frac{2}{r}}}(x){\rm tanh^2}(x)dx=\frac{a^2b}{r^2}A(r),
\end{equation}
where $A(r)$ is a positive number depending on $r\geq1$. From $(\ref{soleq})$, $(\ref{solitary})$ and $(\ref{gensolit1})$ we deduce

\begin{equation}\label{gensolit2}
a=\frac{\sqrt{A(r)b(r^2\omega-b^2)}r}{A(r)b^2},
\end{equation}
with $b$ satisfying the following nonlinear equation

\begin{equation}\label{gensolit3}
- \left( {\frac {\sqrt {A(r)b \left( {r}^{2}\omega-{b}^{2} \right) }r}{A(r){
			b}^{2}}} \right) ^{2\,r+1}{r}^{4}+ \frac{\left( 1+r \right) \sqrt {A(r)b
	\left( {r}^{2}\omega-{b}^{2} \right) }r \left( {r}^{2}+{\frac {
		\left( {r}^{2}\omega-{b}^{2} \right) {r}^{2}}{{b}^{2}}} \right)} {A(r)}=0.
\end{equation}
\indent Let $r\geq1$ be fixed. For a fixed $\omega_0>0$ large enough and $b_0>0$ sufficiently small, we obtain by implicit function theorem the existence of an open interval $I_{\omega_0}$ around $\omega_0$ and an open interval $I_{b_0}$ around $b_0$ such that $(\ref{gensolit3})$ is valid for all $\omega\in I_{\omega_0}$ and $b\in I_{b_0}$. In addition, there exists a unique smooth function $B:I_{\omega_0}\rightarrow I_{b_0}$ such that $B(\omega)=b(\omega)$ in $I_{\omega_0}$. As a consequence, parameter $a$ in $(\ref{gensolit2})$ is well determined and depends smoothly on $\omega\in I_{\omega_0}$. Therefore, at least locally the profile in $(\ref{solitary})$ solves equation $(\ref{soleq})$.\\\\

\noindent {\bf Case 2.} $\mathbb{B}=\mathbb{T}$.\\
\noindent $\bullet$ If $r=1$ we need to find $a$, $b$ and $\omega$ in $(\ref{per1})$ in terms of $k$. To do so, let us consider $k\in (0,k_{*})$, where $k_{*}\approx 0.979653$. This restriction enables us to consider smooth parameters $a$, $b$ and $\omega$ in terms of $k$ given by
\begin{equation}\label{ak}a=\displaystyle\frac{\sqrt{6\pi}K(k)}{\sqrt{8(1-k^2)K(k)^4-4(2-k^2)E(k)K(k)^3+3\pi^3}},\end{equation}

\begin{equation}\label{bk}b=\displaystyle\frac{K(k)}{\pi},
\end{equation} and

\begin{equation}\label{wk}
\omega=\displaystyle\frac{3(2-k^2)\pi K(k)^2}{8(1-k^2)K(k)^4-4(2-k^2)E(k)K(k)^3+3\pi^3}.
\end{equation}
We see from $(\ref{ak})$ and $(\ref{wk})$ that $a$ and $\omega$ must satisfy $a>\frac{\sqrt{2}}{2}$ and $\omega>\frac{1}{2}$. Moreover, both parameters are strictly increasing functions in terms of $k$ according to the Figures 7 and 8.

\begin{figure}[!htb]
	\centering
	\begin{minipage}[b]{0.4\linewidth}
		\includegraphics[scale=0.33]{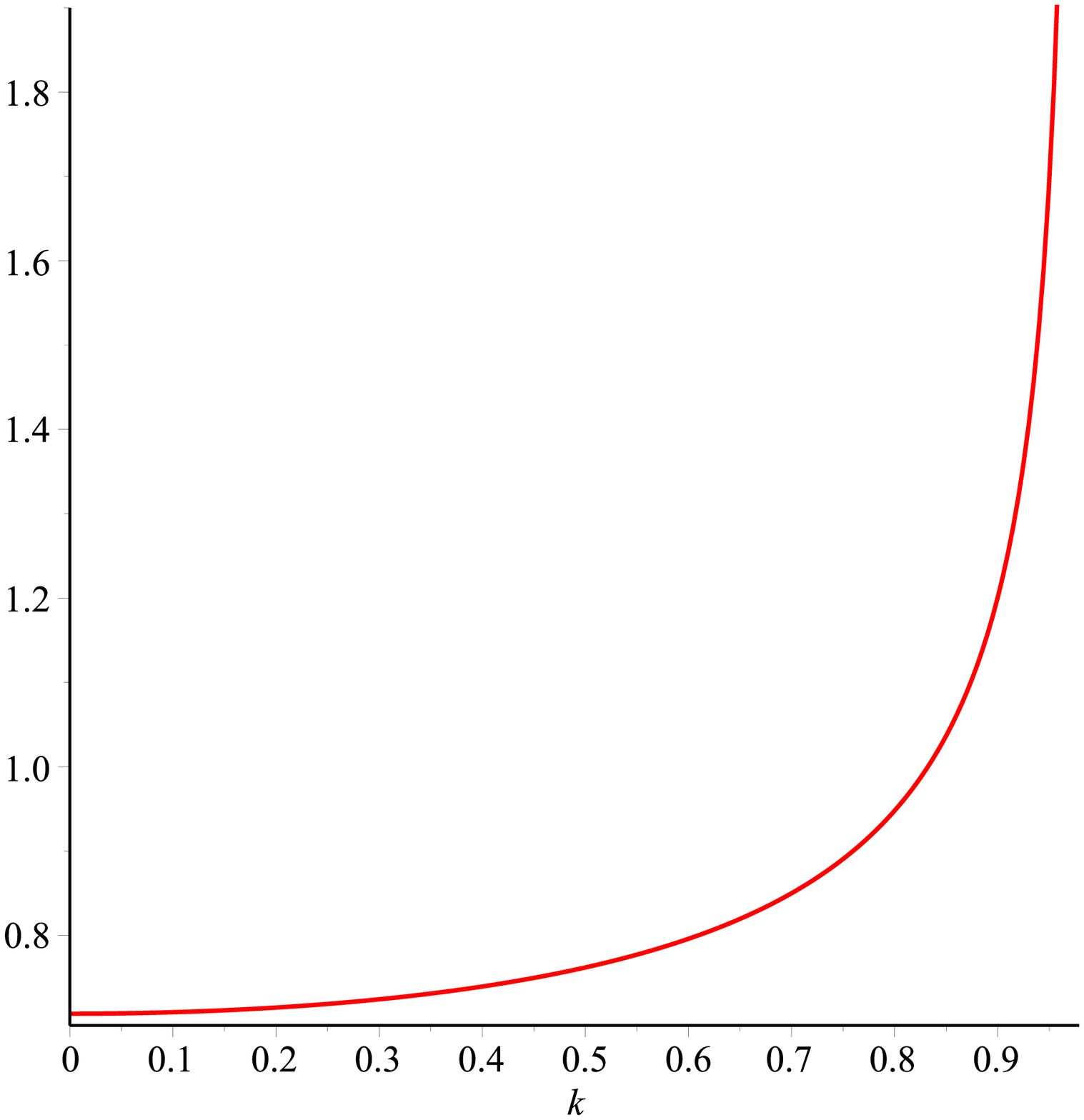}
		\caption{Graphic of $a$.}
	\end{minipage}\hfill
	\begin{minipage}[b]{0.4\linewidth}
		\includegraphics[scale=0.33]{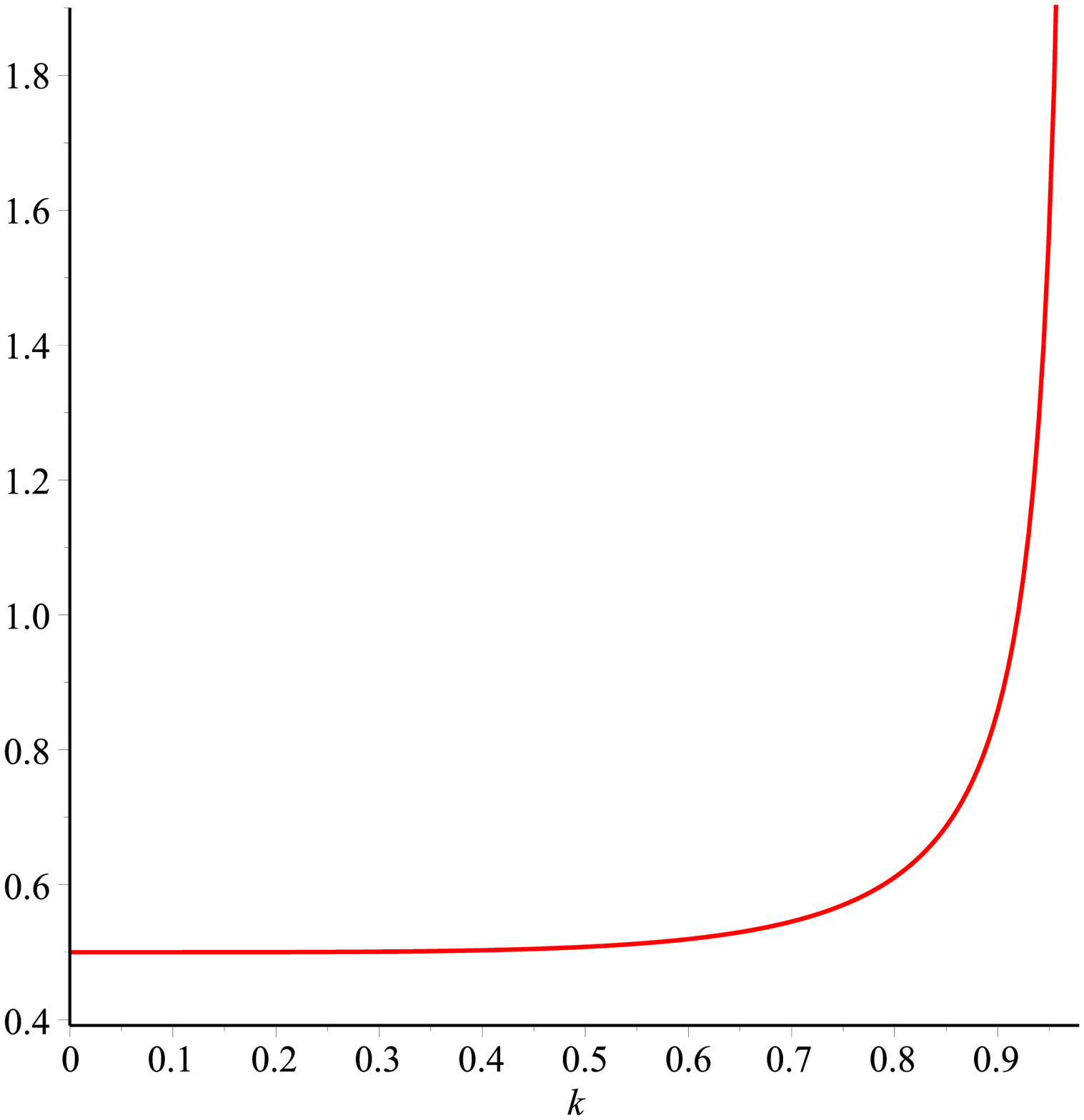}
		\caption{Graphic of $\omega$.}
	\end{minipage}
\end{figure}

\noindent $\bullet$ If $r=2$, we consider
$b=\displaystyle\frac{K(k)}{\pi},$
\begin{equation}\label{alphak}\alpha=-k^2+1-\sqrt{k^4-k^2+1},\end{equation} and \begin{equation}\label{wk4}\omega=\displaystyle\frac{a^4k^2(\alpha k^2-k^2-\alpha)}{\alpha^2(\alpha-2)}.\end{equation}  The denominator in $(\ref{per2})$ makes sense since $\alpha<0$. We omit the expression for $a=a(k)$ to simplify the notation. We can plot $a$ and $\omega$ in terms of $k$ according to the Figures 9 and 10.

\newpage 
\begin{figure}[!htb]
	\centering
	\begin{minipage}[b]{0.4\linewidth}
		\includegraphics[scale=0.33]{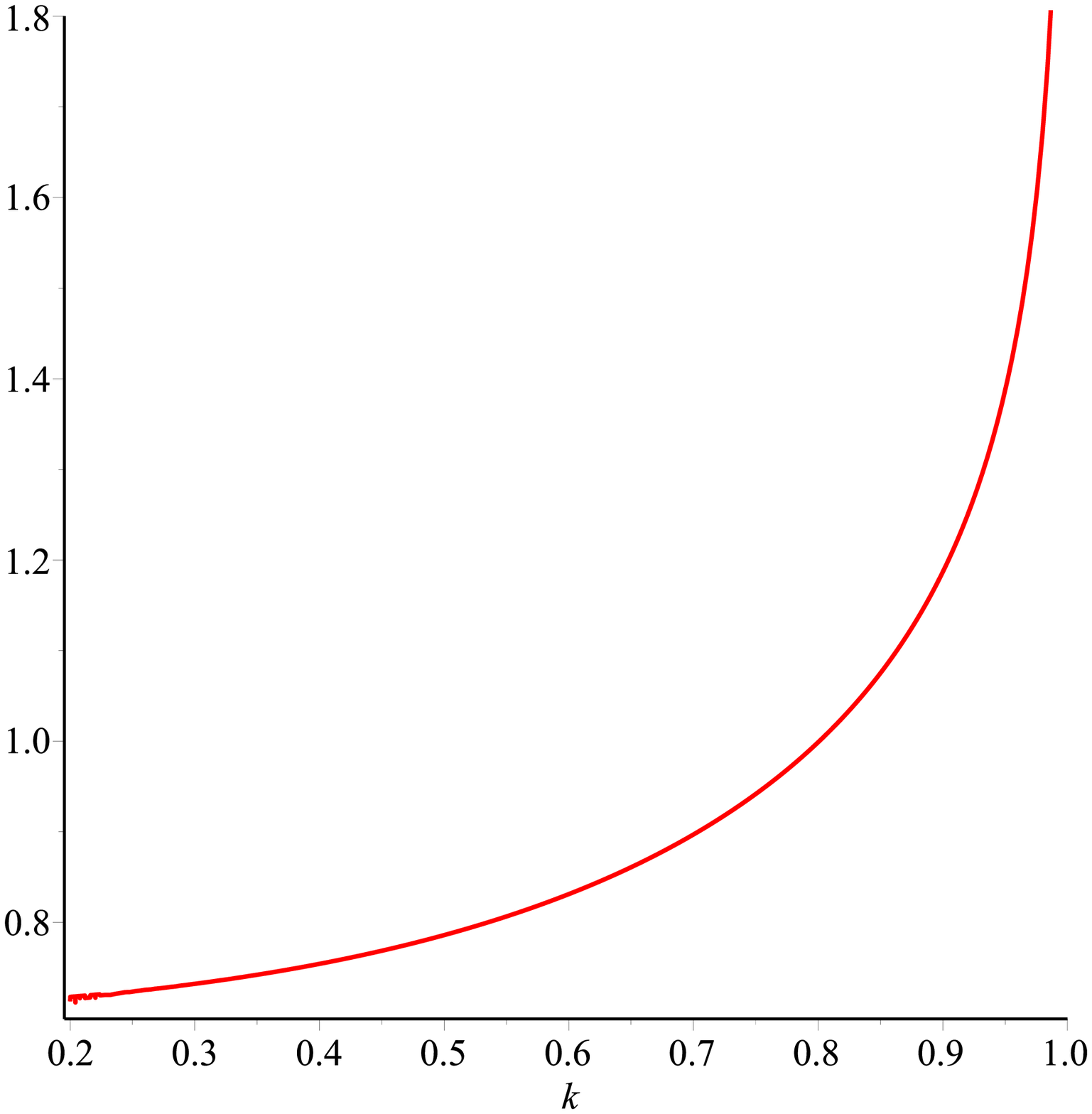}
		\caption{Graphic of $a$.}
	\end{minipage}\hfill
	\begin{minipage}[b]{0.4\linewidth}
		\includegraphics[scale=0.33]{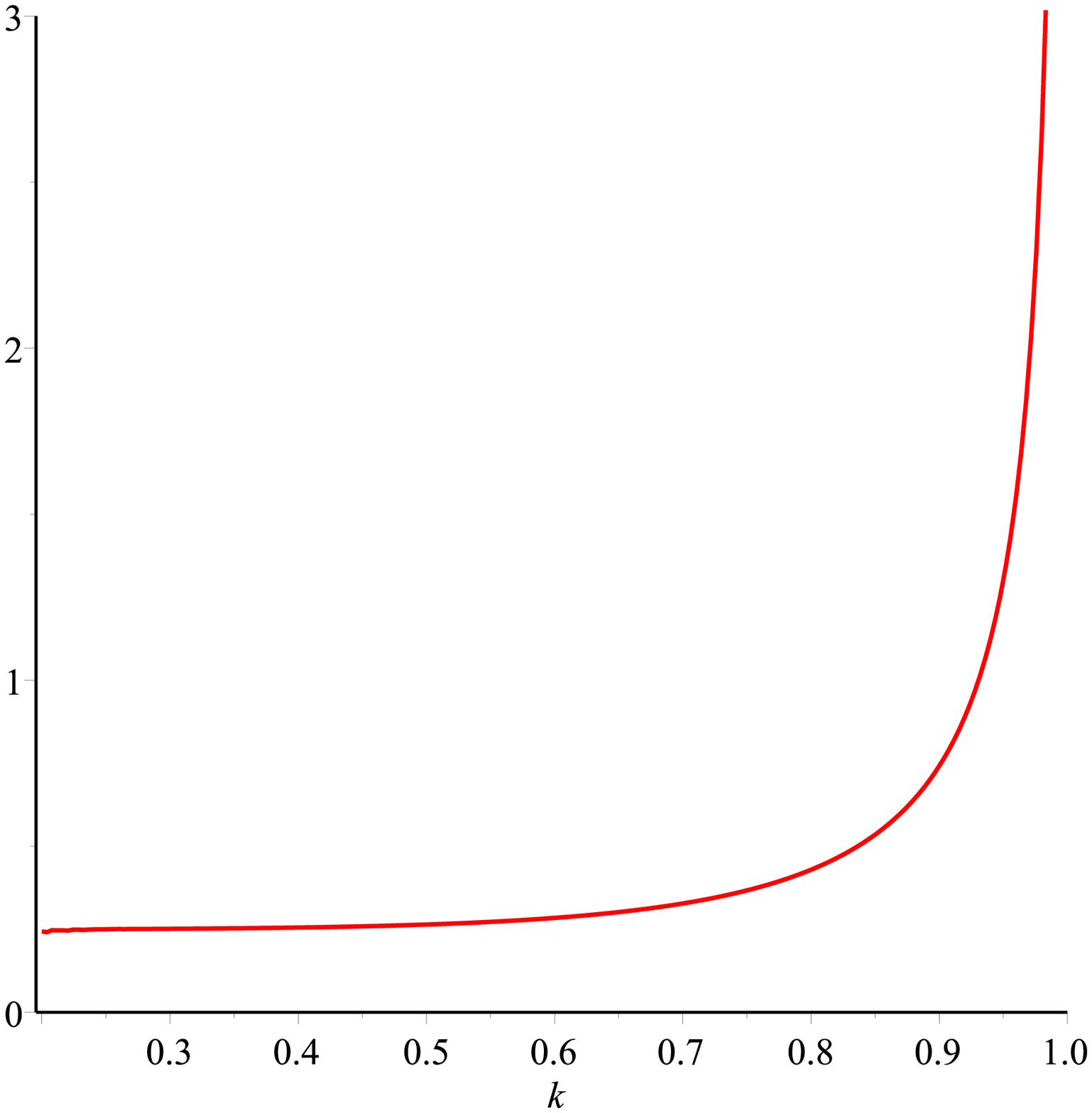}
		\caption{Graphic of $\omega$.}
	\end{minipage}
\end{figure}

\subsection{Local and global well-posedness results}\label{gwp}

\indent We present local and global well-posedness results for the Cauchy problem associated to the equation $(\ref{nls})$ given by
\begin{equation}\label{GNLS1}
\left\{
\begin{aligned}
& iu_{t} +\left(1+\int_{\mathbb{B}}|u_x|^2dx\right)u_{xx}+ |u|^{2r}u =0, ~\hbox{ in }\mathbb{B} \times \R^{+},\\
&u(0)=u_0(x), ~ x\in \mathbb{B}.
\end{aligned}
\right.\end{equation}
\indent  Concerning the case $\mathbb{B}=\mathbb{T}$, global solutions in $H^2(\mathbb{T})$ are determined using the compact embedding $H^s(\mathbb{T})\hookrightarrow L^q(\mathbb{T})$ for $s,q\geq1$ and a convenient Galerkin's approximation combined with good bounds of the approximate solutions. In addition, the compact embedding $H^2(\mathbb{T})\hookrightarrow H^1(\mathbb{T})$ allows us to prove the existence and uniqueness of global solutions in $C([0,+\infty);H^1(\mathbb{T}))$ employing density arguments. Since the proof in the periodic case is standard, we only consider the case $\mathbb{B}=\R$ and solutions in $H^1(\mathbb{R})$.\\

\indent  Let us consider the auxiliary initial value problem associated to the $(\ref{GNLS1})$ given by

 \begin{equation}\label{GNLS2}
 \left\{
 \begin{aligned}
 & iv_{t} +\beta(t)v_{xx}+ |v|^{2r}v =0, ~\hbox{ in }\mathbb{R} \times \R^{+},\\
 &v(x,t_0)=v_0(x), ~ x\in \mathbb{R}.
 \end{aligned}
 \right.\end{equation}
where $t_0\in\R^{+}$ is a fixed real number and $\beta$ is a continuous real-valued function depending on the time $t\in\R^{+}$. For each pair $(t,l)\in\R^{+}\times\R^{+}$ we denote by $S(t,l)$ the propagator associated to the linear part of $(\ref{GNLS2})$. For a fixed $t_0\in\R$ one has
\begin{equation}
S(t,t_0)v_0(x)=
\left(e^{-iB(t,t_0)\xi^2}\hat{v}_0(\xi)\right)^{\widecheck{}}(x),
\label{linprop1}
\end{equation}
where $\hat{f}$ is the Fourier Transform of $f$ in $L^2(\R)$, $\check{f}$ denotes the inverse Fourier Transform in the same space and $B$ is defined by $B(t,l):=\int_{l}^{t}\beta(\tau)d\tau$. Let $t,m,l$ be real numbers. The linear propagator $S(t,l)$ satisfies the following basic properties
\begin{equation}
S(t,l)=S(t,m)S(m,l),\  S(t,l)=S(l,t)^{-1}\ \mbox{and}\ S(t,l)=S(t,0)S(l,0)^{-1}:=S(t)S(l)^{-1}.
\label{prop1}
\end{equation}
 If $s\geq0$ and $t_0$ is a fixed number, the linear propagator $S(t,t_0)$ is an isometry in $H^s(\R)$, that is, for all $f\in H^s(\mathbb{R})$ one has $||S(t,t_0)f||_{H^s}=||f||_{H^s}$.\\

 \indent Similarly to the arguments in \cite{cazenave3}, we can prove the following following results.

 \begin{proposition}
 	Let $(p,q)$, $(p_0,q_0)$ and $(p_1,q_1)$ be any admissible pairs. The linear propagator associated to the linear part of the equation in $(\ref{GNLS2})$ satisfies
 	\begin{equation}
 	\left(\int_{\R^{+}}||S(t,\tau)f||_{L^p}^qdt\right)^{\frac{1}{q}}\leq C_1||f||_{L^2},
 	\label{stric1}
 	\end{equation}
 	\begin{equation}
 	\left(\int_{\R^{+}}\left\|\int_{\R^{+}}S(t,\tau)F(\cdot,\tau)
 	d\tau\right\|_{L^p}^qdt\right)^{\frac{1}{q}}\leq C_2\left(\int_{\R^{+}}||F(\cdot,t)||_{L^{p'}}^{q'}dt\right)
 	^{\frac{1}{q'}},
 	\label{stric2}
 	\end{equation}
and

\begin{equation}
\left(\int_{0}^T\left\|\int_{0}^{t}S(t,\tau)F(\cdot,\tau)
d\tau\right\|_{L^{p_1}}^{q_1}dt\right)^{\frac{1}{q_1}}\leq C_3\left(\int_{0}^{T}||F(\cdot,t)||_{L^{p_0'}}^{q_0'}dt\right)
^{\frac{1}{q_0'}},
\label{stric3}
\end{equation} 	
where for $i=1,2,3$ one has that $C_i>0$ are constants depending on the admissible pairs.
 	\label{strichartz}
 \end{proposition}
\begin{proof}
	See \cite[Theorem 2.3.3]{cazenave3}.
	
\end{proof}

\indent In what follows, we shall consider $t_0=0$ in $(\ref{GNLS2})$ and we restrict the Cauchy problems $(\ref{GNLS1})$ and $(\ref{GNLS2})$ to the case $(x,t)\in \R\times [0,+\infty)$. Using the estimates in the last proposition, we can prove next result employing a fixed point argument.
\begin{proposition}\label{teo11}
	For all $v_0\in H^1(\mathbb{R})$ there exists $T_{max}>0$
	and a unique solution $v$ related to the equation $(\ref{GNLS2})$
	 such that
	$v\in C([0,T_{max});H^1(\mathbb{R}))\cap C^1([0,T_{max});H^{-1}(\mathbb{R})).$ In addition, there is a blow-up alternative in the sense that if $T_{max}<0$, then $||v(t)||_{H^1}\rightarrow +\infty$ as $t\rightarrow T_{\max}$. The solution in fact has a smoothing effect in the sense that $v\in L^{q}([0,T_{max});W^{1,p}(\R))$, where $(p,q)$ is an admissible pair.
	
\end{proposition}

\begin{proof}
	The proof of this proposition is similar to \cite[Theorem 4.8.1]{cazenave3}.
\end{proof}

\indent  We are in position to consider  $\beta(t)=1+\int_{\R}|u_x(x,t)|^2dx$ in the equation $(\ref{GNLS2})$ and according with the fixed point argument used in Proposition $\ref{teo11}$ we see that  $\beta$ is continuous in time. Proposition $\ref{teo11}$ establishes the existence of local solutions in $H^1$ associated to the Cauchy problem $(\ref{GNLS1})$.

\begin{proposition}\label{teo13}
	For all $u_0\in H^1(\mathbb{R})$ there exists $T_{max}>0$
	and a unique solution $u$ related to the equation $(\ref{GNLS1})$
	such that
	$u\in C([0,T_{max});H^1(\mathbb{R}))\cap C^1([0,T_{max});H^{-1}(\mathbb{R})).$ In addition, there is a blow-up alternative in the sense that if $T_{max}<0$, then $||u(t)||_{H^1}\rightarrow +\infty$ as $t\rightarrow T_{\max}$. 	
\end{proposition}

\begin{proof}
	See Proposition $\ref{teo11}$.
\end{proof}

\indent We finish this subsection with the following theorem which provides us the existence of global solutions in $H^1(\mathbb{B})$.

\begin{theorem}\label{propositionexist}
For $r\in[1,4)$ local solutions in $H^1$ for the Cauchy problem $(\ref{GNLS1})$ are global in time in the sense that $T$ can be chosen as $T=+\infty$. If $r=4$ and $||u_0||_{L^2}$ is small enough, the solution is also global in time.
\end{theorem}

\begin{proof} To prove the global theory, we employ the Gagliardo-Nirenberg inequality to estimate the last term of the identity $(\ref{E})$ in terms of  $||u_x||_{L^2}$ as
	\begin{equation}\label{GN1}
	||u||_{L^{2(r+1)}}\leq C_1||u_x||_{L^2}^{\frac{r}{2(r+1)}}||u||_{L^2}^{\frac{r+2}{2(r+1)}}+C_2||u||_{L^2},
	\end{equation}
	where $C_1>0$ and $C_2\geq0$ are constants with $C_2=0$ when $\mathbb{B}=\R$.

Since the $L^2$-norm is a conserved quantity for $(\ref{nls})$, we can assume that $(\ref{GN1})$ can be given in terms of our local solution determined in Proposition $\ref{teo11}$ and we rewrite it as

\begin{equation}\label{GN2}
||u||_{L^{2(r+1)}}\leq C_1||u_x||_{L^2}^{\frac{r}{2(r+1)}}||u_0||_{L^2}^{\frac{r+2}{2(r+1)}}+C_2||u_0||_{L^2},
\end{equation}
where we omit the temporal variable to simplify the notation.\\
\indent By $(\ref{E})$ we have

\begin{equation}\label{GN3}\begin{array}{llll}||u_x||_{L^2}^2+||u_x||_{L^2}^4&\leq& 2E(u_0)+\frac{1}{r+1}||u||_{L^{2r+2}}^{2r+2}\\\\
&\leq&2E(u_0)+ C_3||u_x||_{L^2}^r||u_0||_{L^2}^{r+2}+C_4(||u_0||_{L^2}). \end{array}\end{equation}
Therefore if $r\in[1,4)$ the solutions are global in time as requested. When $r=4$, we need to assume a convenient smallness on the initial data in the $L^2-$norm to get global solutions.

\end{proof}
\subsection{The definition of orbital stability}\label{deforbstab}

 In this subsection, we define our notion of orbital stability.  Since  \eqref{hamiltonian} is invariant by the transformations \eqref{l0} and \eqref{l00}, we define the orbit generated by $\Phi=(\phi,0)$ as
\begin{equation}\label{l1}
\begin{split}
\Omega_\Phi&=\{T_1(\theta) T_2(s)\Phi;\;\;\theta,s\in\R\}\\
&= \left\{ \left(\begin{array}{cc}
\cos\theta & \sin\theta\\
-\sin\theta & \cos\theta
\end{array}\right)\left(\begin{array}{c}
\phi(\cdot-s)\\
0
\end{array}\right);\;\;\theta,s\in\R  \right\}.
\end{split}
\end{equation}
Over $\Hh^1$, we define the pseudo-metric $d$ given by
$$
d(f,g):=\inf\{\|f-T_1(\theta)
T_2(s)g\|_{\Hh^1};\;\theta,s\in\R\}.
$$
\indent By definition, the distance between $f$ and $g$ is the distance between $f$ and the orbit generated by  $g$ under the action of rotations and translations. In particular,
\begin{equation}\label{l2}
d(f,\Phi)=d(f,\Omega_\Phi).
\end{equation}

\begin{definition}\label{stadef}
	Let $\Theta(x,t)=(\phi(x)\cos(\omega t), \phi(x)\sin(\omega t))$ be a standing wave for \eqref{hamiltonian}. We say that $\Theta$ is orbitally stable in $\Hh^1$ provided that, given $\ve>0$, there exists $\delta>0$ with the following property: if $U_0\in \Hh^1$ is an initial data associated to the Cauchy problem $(\ref{hamiltonian})$ and satisfying $\|U_0-\Phi\|_{\Hh^1}<\delta$, then the solution, $U(t)$, of \eqref{hamiltonian} with initial condition $U_0$ exist for all $t\geq0$ and satisfies
	$$
	d(U(t),\Omega_\Phi)<\ve, \qquad \mbox{for all}\,\, t\geq0.
	$$
	Otherwise, we say that $\Theta$ is orbitally unstable in $\Hh^1$.
\end{definition}

\section{Spectral Analysis} \label{sepecsec}

In this section we are going to use basic facts about Sturm-Liouville and Floquet theories to know the quantity and multiplicity of the non-positive eigenvalues of $\mathcal{L}$. Since $\mathcal{L}$ is a diagonal operator, its eigenvalues are given by the eigenvalues of the operators $\Lum$ and $\Ldois$.

\subsection{The spectrum of $\Lum$ and $\Ldois$ - case $\mathbb{B}=\mathbb{R}$.}\label{spectcaseR}
Attention will be turned to the spectrum of the operator $\Lum$ and $\Ldois$ for the case of solitary waves $(\ref{solitary})$. First, we present the spectral analysis for the operator $\Lum$.

\begin{proposition}\label{specL1}
The operator $\Lum$ in \eqref{L1} defined on $L^2(\R)$ with domain $H^2(\R)$
has a unique negative eigenvalue, which is simple with positive associated eigenfunction. The eigenvalue zero is
simple with associated eigenfunction $\phi'$. Moreover the rest of the
spectrum is bounded away from zero and the essential spectrum is the interval
$\left[\frac{\omega}{1+\int_{\mathbb{R}}\phi'^2dx},+\infty\right)$.
\end{proposition}
\begin{proof}
The second part of the proposition can be established using Weyl's Criterium in \cite[Theorem B.48]{an} and we are going to prove the first part for the case $r=1$ since the cases $r=2$ and $r=4$ are quite similar. Indeed, we obtain by the explicit profile in $(\ref{solitary})$ and some additional calculations that
$$(\Lum \phi,\phi)_{L^2}=-2\int_{\mathbb{R}}\phi^4dx+2\left(\int_{\mathbb{R}}\phi'^2dx\right)^2<0.$$
In other words, $\Lum$ has at least one negative eigenvalue. Next, for any $P\in
H^2(\mathbb{R})$ we have
\begin{equation}\label{quadr1}
(\Lum P,P)_{L^2}=(\mathcal{L}_{1}P,P)_{L^2}
+2(\phi',\partial_xP)_{L^2}^2,
\end{equation}
where $\mathcal{L}_1=
-\left(1+\int_{\mathbb{R}}\phi'^2dx\right)\partial_x^2+\omega-3\phi^2$. Since $1+\int_{\mathbb{R}}\phi'^2dx$ is a positive constant, $\phi'$ has only one zero over $\mathbb{R}$ and $\Lum\phi'=\mathcal{L}_1\phi'=0$, we see from the classical Sturm-Liouville Theory that $\mathcal{L}_1$ has only one negative eigenvalue which is simple and zero is a simple eigenvalue with associated eigenfunction $\phi'$. Moreover, the remainder of the spectrum of $\mathcal{L}_1$ is bounded away from zero.\\
\indent Let $\chi_1$ be the eigenfunction associated to the unique
negative eigenvalue  of $\mathcal{L}_{1}$. Then, if $g\bot\chi_1$, it follows that
$(\mathcal{L}_{1}g,g)_{L^2}\geq0$. Therefore, if
$\lambda$ denotes the second eigenvalue of
$\Lum$, from the Min--Max Principle (see
\cite[Theorem XIII.1]{RS}), we obtain using $(\ref{quadr1})$
$$\displaystyle\lambda=\max_{\chi}\min_{g\bot[\chi],||g||=1}(\Lum g,g)_{L^2}
\geq\min_{g\bot[\chi_1], ||g||=1}(\mathcal{L}_1 g,g)_{L^2}\geq0.$$
 This proves that the first eigenvalue of $\Lum$ is negative and simple and zero is the second eigenvalue. To prove that $0$ is also simple, we use a similar analysis  as above to prove that the third eigenvalue of $\Lum$ is positive.
\end{proof}
\indent Concerning $\Ldois$, we have the following result.

\begin{proposition}\label{specL2}
The operator $\Ldois$ in \eqref{L2} defined on $L^2(\R)$ with domain
$H^2(\R)$ has no negative eigenvalue. The eigenvalue zero is simple with
associated eigenfunction $\phi$. Moreover the rest of the spectrum is bounded
away from zero and the essential spectrum is the interval $\left[\frac{\omega}{1+\int_{\mathbb{R}}\phi'^2dx},+\infty\right)$.
\end{proposition}
\begin{proof}
	Since $\phi$ is positive and $\Ldois\phi=0$, we see directly from the Sturm-Liouville Theory that $0$ is the first eigenvalue of $\Ldois$ and it is result to be simple.
\end{proof}

We finish this subsection by stating the spectral properties of the ``linearized'' operator $\mathcal{L}$. Indeed, a combination of Propositions \ref{specL1} and \ref{specL2} gives us the following.

\begin{theorem}\label{specLline}
The operator $\mathcal{L}$ in \eqref{L} defined on $\Ll^2$
with domain $\Hh^2$ has a unique negative eigenvalue, which
is simple. The eigenvalue zero is double with associated eigenfunctions
$(\phi',0)$ and $(0,\phi)$. Moreover the essential spectrum is the interval
$\left[\frac{\omega}{1+\int_{\mathbb{R}}\phi'^2dx},+\infty\right)$.
\end{theorem}

\begin{proof}
	The proof of this result is an immediate consequence of Propositions $\ref{specL1}$ and $\ref{specL2}$.
\end{proof}

\subsection{The spectrum of $\Lum$ and $\Ldois$ - case $\mathbb{B}=\mathbb{T}$.}\label{spectcaseT}

We start this subsection by establishing some basic facts on the Floquet theory. In general setting, we consider the second order ordinary differential equation as
\begin{equation}
- \psi'' + g(\omega,\psi) = 0, \label{ode} \end{equation}
where $g$ is a smooth function in all variables. We assume that the parameter $\omega$ belongs to an open set
$\mathcal{P} \subset\mathbb{R}$.\\
\indent Now, let $\mathcal{D}$ be the linearized equation of
(\ref{ode}) at $\psi$, where $\psi$ is an $2\pi-$periodic
solution of (\ref{ode}). The linearized operator  around $\psi$
\begin{equation}
\mathcal{D}y = - y'' + g'(\omega, \psi)\, y,
\;\;\; \omega \in \mathcal{P} \label{hill} \end{equation} is a Hill
operator, therefore, according to \cite{magnus},
the spectrum of $\mathcal{D}$ is formed by an unbounded
sequence of real numbers
$$
\lambda_0<\lambda_1\leq \lambda_2 < \lambda_3\leq \lambda_4\cdot\cdot\cdot <\lambda_{2n-1}\leq \lambda_{2n}<\cdot\cdot\cdot,
$$
where equality means that $\lambda_{2n-1}=\lambda_{2n}$ is a double eigenvalue.

\indent According with the \textit{Oscillation Theorem} in \cite{magnus}, the spectrum of $\mathcal{D}$ is
characterized by the number of zeros of the eigenfunctions. In fact, if $h$
is an eigenfunction associated to the eigenvalue $\lambda_{2n-1}$ or
$\lambda_{2n}$, then $h$  has exactly $2n$ zeros in the half-open
interval $[0, \; 2\pi)$. \\
\indent Let $\{y_1,y_2\}$ be a fundamental set related to the Hill equation
\begin{equation}\label{fundset}
- y'' + g'(\omega, \psi)=0.
\end{equation}
\indent Suppose that $y_1$ is an $2\pi-$periodic solution of $(\ref{fundset})$. The arguments in \cite{magnus} establish a connection between $y_1$ and $y_2$ through the equality,
\begin{equation}\label{relpq}
y_2(x+2\pi)=y_2(x)+\theta y_1(x),
\end{equation}
where $\theta$ is a real constant.\\
\indent Next result gives us a necessary and sufficient condition to decide about the periodicity of $y_2$ and a sufficient condition to know the exact position of the zero eigenvalue associated to $\mathcal{D}$ when $\theta\neq0$.
\begin{theorem}\label{teo12}
	Let $y_1$ be the eigenfunction of $\mathcal{D}$ in $(\ref{hill})$ associated to the zero eigenvalue, and $\theta$ is
	the constant given by $(\ref{relpq})$. The zero eigenvalue  is simple if and only if $\theta \neq 0$. Moreover, if $y_1$
	has $2n$ zeros in the half-open interval $[0,2\pi)$, then $\lambda_{2n-1}=0$ if $\theta < 0$, and $ \lambda_{2n}=0$ if $\theta > 0$.
\end{theorem}
\begin{proof}
	See \cite{nn}.
\end{proof}

We have similar results as established by Propositions $\ref{specL1}$ and $\ref{specL2}$ in the periodic context.

\begin{proposition}   \label{propRe}
	 Operator $\Lum$ in $(\ref{L1})$
		defined in $L^2(\mathbb{T})$ with domain
		$H^2(\mathbb{T})$ has  its first
		two eigenvalues simple, being the eigenvalue zero the second one
		with eigenfunction $\phi'$. Moreover, the remainder of the spectrum is constituted by a discrete
		set of eigenvalues.
\end{proposition}
\begin{proof}
	The proof is very similar to the proof of Proposition $\ref{specL1}$. The main difference is that we need to use Theorem $\ref{teo12}$ instead of the Sturm-Liouville Theory in order to obtain the behavior of the first two eigenvalues of $\mathcal{L}_1$. As usual, the second part of the proposition can be established using Weyl's Criterium in \cite[Theorem B.48]{an} and we are going to prove the first part for the case $r=1$.
	
	First of all, we have that
\begin{eqnarray}\label{normder}\displaystyle\int_{\mathbb{T}}\phi'^2dx&=&2a^2bk^4\displaystyle\int_0^{K(k)}\textrm{cn}^2(x,k)\ \textrm{sn}^2(x,k)\ dx\nonumber\\
	&=&\displaystyle\frac{-8(1-k^2)K(k)^4+4(2-k^2)E(k)K(k)^3}{8(1-k^2)K(k)^4-4(2-k^2)E(k)K(k)^3+3\pi^3}=:\tau_1(k),\end{eqnarray}
	
	and
	
	\begin{equation}\label{IntPhi4}
\displaystyle\int_{\mathbb{T}}\phi^4 dx=\displaystyle\frac{24\pi^3K(k)^3[2(2-k^2)E(k)-(1-k^2)K(k)]}{[8(1-k^2)K(k)^4-4(2-k^2)E(k)K(k)^3+3\pi^3]^2}=:\tau_2(k).	
	\end{equation}

	So, by using (\ref{normder}) and (\ref{IntPhi4}), we obtain
\begin{eqnarray*}(\Lum \phi,\phi)_{L^2}=-2\int_{\mathbb{T}}\phi^4dx+2\left(\int_{\mathbb{T}}\phi'^2dx\right)^2=-2\tau_2(k)+2[\tau_1(k)]^2=:\tau(k).\end{eqnarray*}	

The picture below shows us that $\tau(k)<0$, for all $k\in (0,k_{*})$:

\begin{figure}[!htb]
	\includegraphics[scale=0.35]{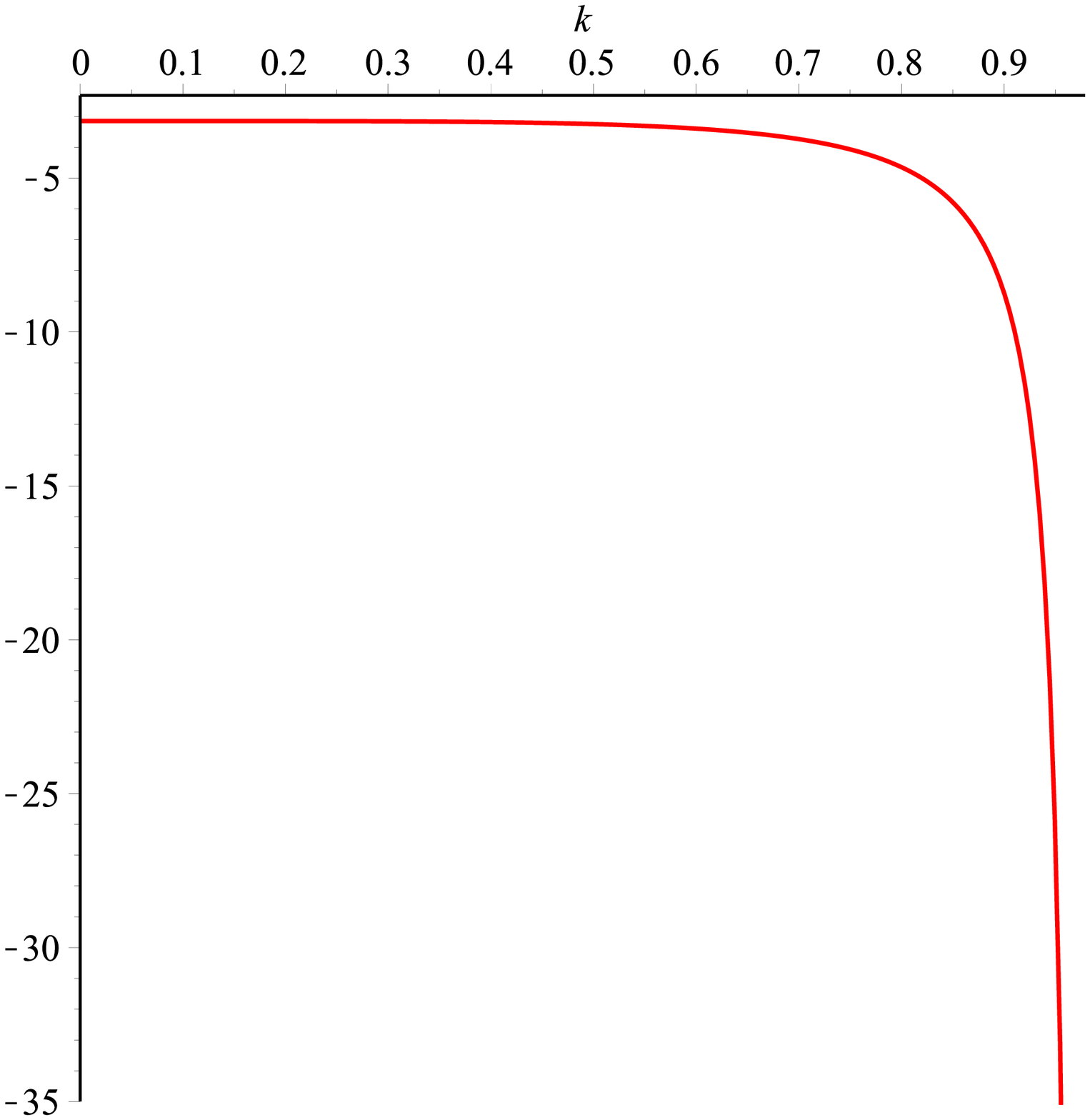}

\end{figure}
		
Thus, we obtain
	$$(\Lum \phi,\phi)_{L^2}<0,$$
	that is, $\Lum$ has at least one negative eigenvalue. On the other hand, for any $P\in
	H^2(\mathbb{T})$, we have
	\begin{equation}\label{quadr12}
	(\Lum P,P)_{L^2}=(\mathcal{L}_{1}P,P)_{L^2}
	+2(\phi',\partial_xP)_{L^2}^2,
	\end{equation}
	where $\mathcal{L}_1=
	-\left(1+\int_{\mathbb{R}}\phi'^2dx\right)\partial_x^2+\omega-3\phi^2$. Since $1+\int_{\mathbb{R}}\phi'^2dx$ is a positive constant, $\phi'$ has two zeroes over $\mathbb{T}$ and $\Lum\phi'=\mathcal{L}_1\phi'=0$, we need to apply Theorem $\ref{teo12}$ to decide about the position and the simplicity of the zero eigenvalue. In fact, first we see that the Oscillation Theorem in \cite{magnus} gives us that $0$ can be the second or third eigenvalue of $\mathcal{L}_1$. To determine the exact position, we need to solve the following initial value problem
	
\begin{equation}	\left\{
	\begin{array}{l}
		\displaystyle- \bar{y}'' + \frac{\omega}{1+\int_{\mathbb{T}}\phi'^2dx} \bar{y}-\frac{3\phi^2}{1+\int_{\mathbb{T}}\phi'^2dx}\bar{y} = 0 \\\\
		\bar{y}(0) = - \frac{1}{\phi''(0)} \\\\
		\bar{y}'(0)=0,
	\end{array} \right.
	\label{y} \end{equation}
	where $\{\phi',\bar{y}\}$ is the fundamental set associated to the equation $\mathcal{L}_1 y=0$. Since $\phi'$ is odd and $\bar{y}$ is even, the value of $\theta$ in Theorem $\ref{teo12}$ is given by
	\begin{equation}
	 \theta=
	\frac{ \bar{y}'(2\pi)}{\phi''(0)}.
	\end{equation}
	
	We can solve numerically the initial value problem in $(\ref{y})$ for a fixed  value for $\omega$ in the parameter regime. For instance, if $k_0=0.5$ one can see that $\omega_0=\omega(k_0)\approx 0.508$ and $\theta=-18.7569<0$. Since $\mathcal{L}_1$ is isoinertial with respect to $\omega$ (that is, the quantity and multiplicity of non-positive eigenvalues of $\mathcal{L}_1$ remains constant with $\omega$ in the parameter regime of solutions), we deduce from the arguments in \cite{nn} that $n(\mathcal{L}_1)=1$ and $\ker(\mathcal{L}_1)=[\phi']$, where $n(\mathcal{A})$ indicates the number of negative eigenvalues of a linear operator $\mathcal{A}$.\\
	\indent It remains to consider the case $r=2$. The procedure is similar as in the case $r=1$. Initially, we define a new function in terms of the modulus $k$,
	$$(\Lum \phi,\phi)_{L^2}=-4\displaystyle\int_{\mathbb{T}}\phi^6 dx+2\displaystyle\left(\displaystyle\int_{\mathbb{T}}\phi'^2 dx\right)^2:=\gamma(k).$$ In order to simplify the notation, we omit the (heavy) explicit expression for $\gamma$ in terms of $k$, but we present a picture to describe  its behavior.
\begin{figure}[!htb]
	\includegraphics[scale=0.35]{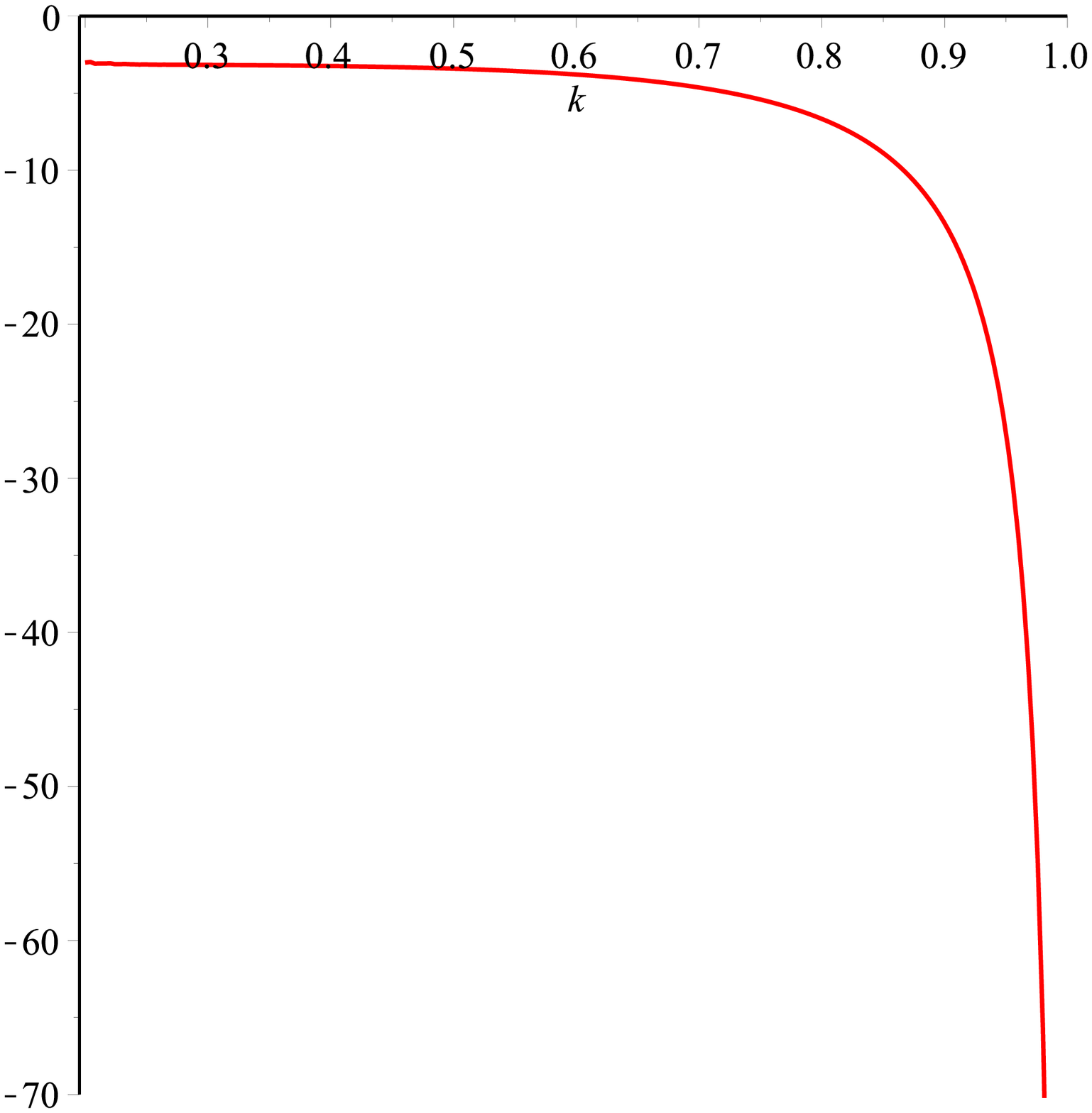}

\end{figure}
	\newpage
	So,
	$$(\Lum \phi,\phi)_{L^2}<0,$$
	and $\Lum$ has at least one negative eigenvalue. To determine the exact position of the zero eigenvalue, we need to solve a similar initial value problem $(\ref{y})$ given by
	
	\begin{equation}	\left\{
	\begin{array}{l}
	\displaystyle- \bar{y}'' + \frac{\omega}{1+\int_{\mathbb{T}}\phi'^2dx} \bar{y}-\frac{5\phi^4}{1+\int_{\mathbb{T}}\phi'^2dx}\bar{y} = 0 \\\\
	\bar{y}(0) = - \frac{1}{\phi''(0)} \\\\
	\bar{y}'(0)=0.
	\end{array} \right.
	\label{y1} \end{equation}\
	
	As above, we can determine $(\ref{y1})$ for a fixed value $\omega=\omega_0$ using numerical calculations to obtain an exact value of $\theta$. For $k_0=0.5$, one has 	
	$\omega_0=\omega(k_0)\approx 0.2642$, and consequently $\theta\approx -40.5143<0$. The fact that $\Lum$ is isoinertial in terms of $\omega$ gives us that $n(\Lum)=1$ and $\ker(\Lum)=[\phi']$ for all values $\omega$ in the parameter regime.
\end{proof}

\begin{proposition} \label{propIm} Operator $\Ldois$ in $(\ref{L2})$
	defined in $L^2(\mathbb{T})$ with domain
	$H^2(\mathbb{T})$ has zero as the first eigenvalue which is simple and the corresponding eigenfunction $\phi$. Moreover, the remainder of the spectrum is constituted by a discrete
	set of eigenvalues.
\end{proposition}

\begin{proof}
	The proof is similar to the Proposition $\ref{specL2}$. In fact, since $\phi$ is positive and $\Ldois\phi=0$, we see directly from the Floquet Theory that $0$ is the first eigenvalue of $\Ldois$ and it is result to be simple.
\end{proof}

Gathering the results in the last two propositions, we finally obtain:
\begin{theorem}\label{specLper}
	The operator $\mathcal{L}$ in \eqref{L} defined on $\Ll^2$
	with domain $\Hh^2$ has a unique negative eigenvalue, which
	is simple. The eigenvalue zero is double with associated eigenfunctions
	$(\phi',0)$ and $(0,\phi)$. Moreover the essential spectrum is empty and the remainder of the spectrum is constituted by a discrete
	set of eigenvalues.
\end{theorem}

\begin{proof}
	The proof of this result is an immediate consequence of Propositions $\ref{propRe}$ and $\ref{propIm}$.
\end{proof}

\section{Orbital stability - Proof of Theorem $\ref{maintheo}$}\label{stasec}

We are going to use \cite{gss1}, \cite{gss2} and \cite{np} (see also \cite{we}) in order to obtain the orbital stability by combining the results determined in subsection $\ref{reviewsec}$ with Theorems $\ref{specLline}$ and $\ref{specLper}$. More specifically, in subsection $\ref{reviewsec}$, we establish the existence of a smooth curve $\omega\in I\subset\mathbb{R}\mapsto\phi\in H^n(\mathbb{B})$, $n\in\mathbb{N}$ of standing waves. Theorems $\ref{specLline}$ and $\ref{specLper}$  give us that $n(\mathcal{L})=1$ and $\ker(\mathcal{L})=[(\phi',0),(0,\phi)]$ and both are sufficient conditions to determine the orbital stability for the cases $r=1,2$.\\
\indent Let us determine $\eta\in H^2(\mathbb{B})$ such that $\Lum\eta=\phi$. This fact is very useful to calculate  $\mathcal{I}=(\Lum\eta,\eta)_{L^2}<0$. In fact,  to obtain $\eta$, we derive equation $(\ref{soleq})$ with respect to $\omega$ to get
\begin{equation}\label{Leta}
-\left(1+\int_{\mathbb{B}}\phi'^2dx\right)
\eta''-2\left(\int_{\mathbb{B}}\phi'
\eta'\ dx\right)\phi''
+\omega\eta-
(2r+1)\phi^{2r}\eta=\phi,
\end{equation}
where $\eta=-\frac{d}{d\omega}\phi$. Thus $\Lum\eta=\phi$ and $\mathcal{I}$ becomes

\begin{equation}
\label{IderL2}
\mathcal{I}=(\Lum\eta,\eta)_{L^2}=-\frac{1}{2}\frac{d}{d\omega}\int_{\mathbb{B}}\phi^2dx.
\end{equation}
To obtain the orbital stability, it makes necessary to determine that $\frac{d}{d\omega}\int_{\mathbb{B}}\phi^2dx>0$ in each case. When $\frac{d}{d\omega}\int_{\mathbb{B}}\phi^2dx<0$, we can prove the orbital instability in the space $\mathbb{H}_e^1$. \\

\noindent {\bf Case 1.} $\mathbb{B}=\R$.\\
 According to $(\ref{IderL2})$, we need to analyze the behavior of $\int_{\mathbb{R}}\phi^2dx$. In fact, from $(\ref{solitary})$ we see that

\begin{equation}
\label{normL2}\int_{\mathbb{R}}\phi^2dx=\frac{a^2}{b}\int_{\mathbb{R}}{\rm sech}^{\frac{2}{r}}(x)dx=\frac{a^2}{b}M(r),
\end{equation}
where $M(r)$ does not depend on $\omega$. To analyze the sign of $\mathcal{I}$, we use the arguments in Subsection $\ref{reviewsec}$ by plotting $\frac{a^2}{b}$ in each case.

\noindent $\bullet$ $r=1$

  \begin{figure}[!htb]
  	 		\includegraphics[scale=0.35]{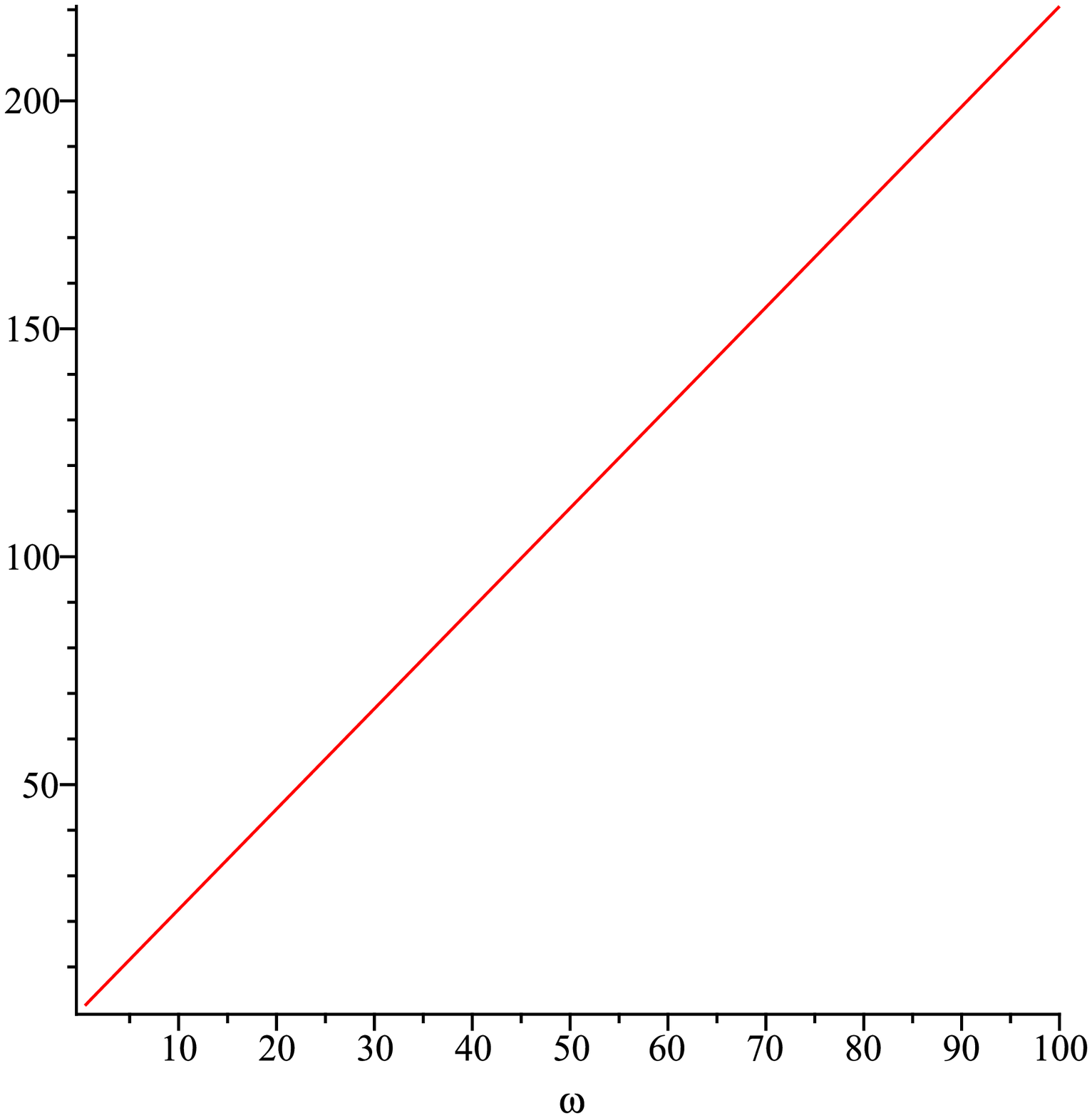}
  		  	\end{figure}

\noindent $\bullet$ $r=2$

\begin{figure}[!htb]
	\includegraphics[scale=0.35]{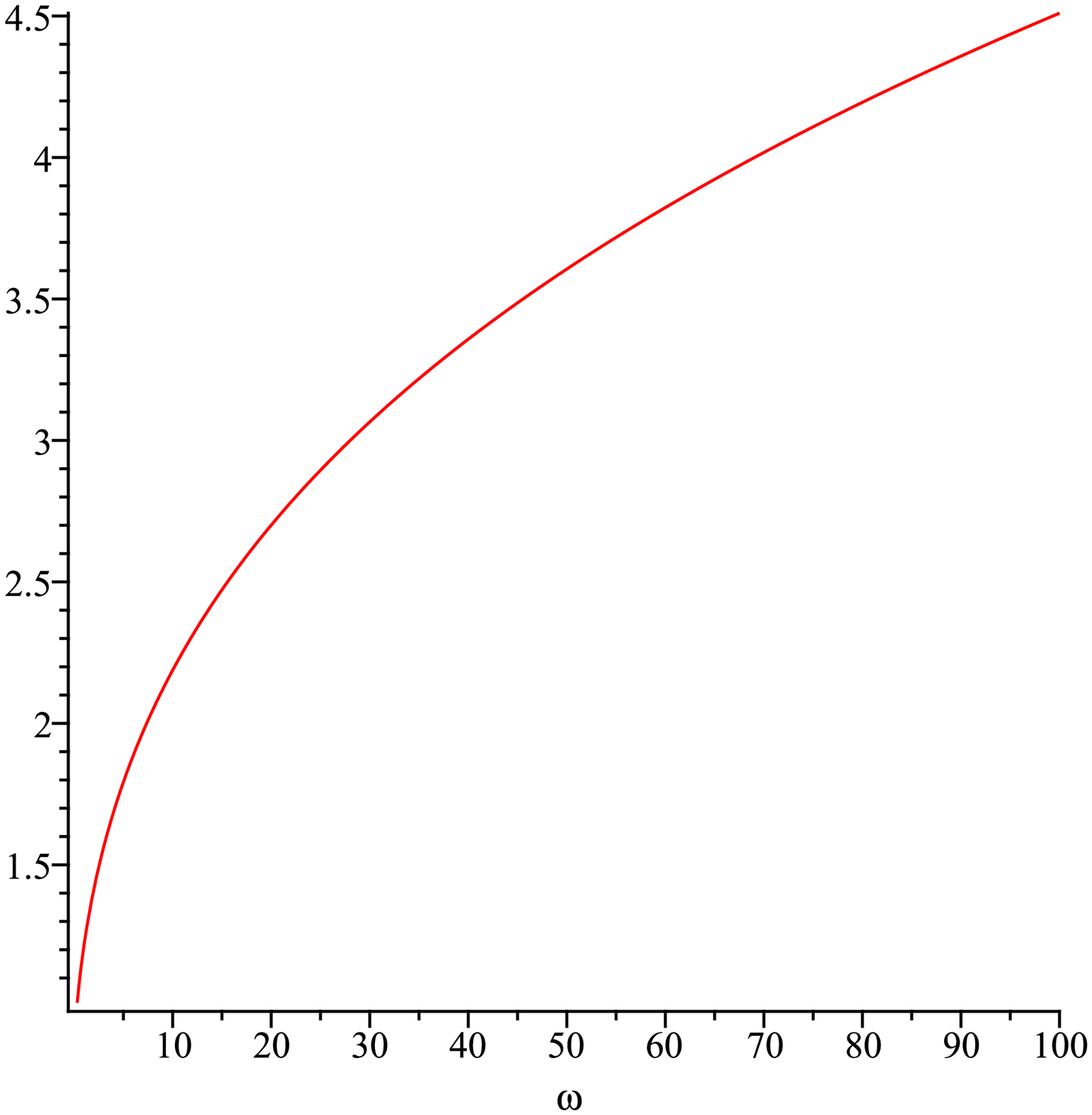}
\end{figure}

\newpage
\noindent $\bullet$ $r=4$

\begin{figure}[!htb]
	\includegraphics[scale=0.35]{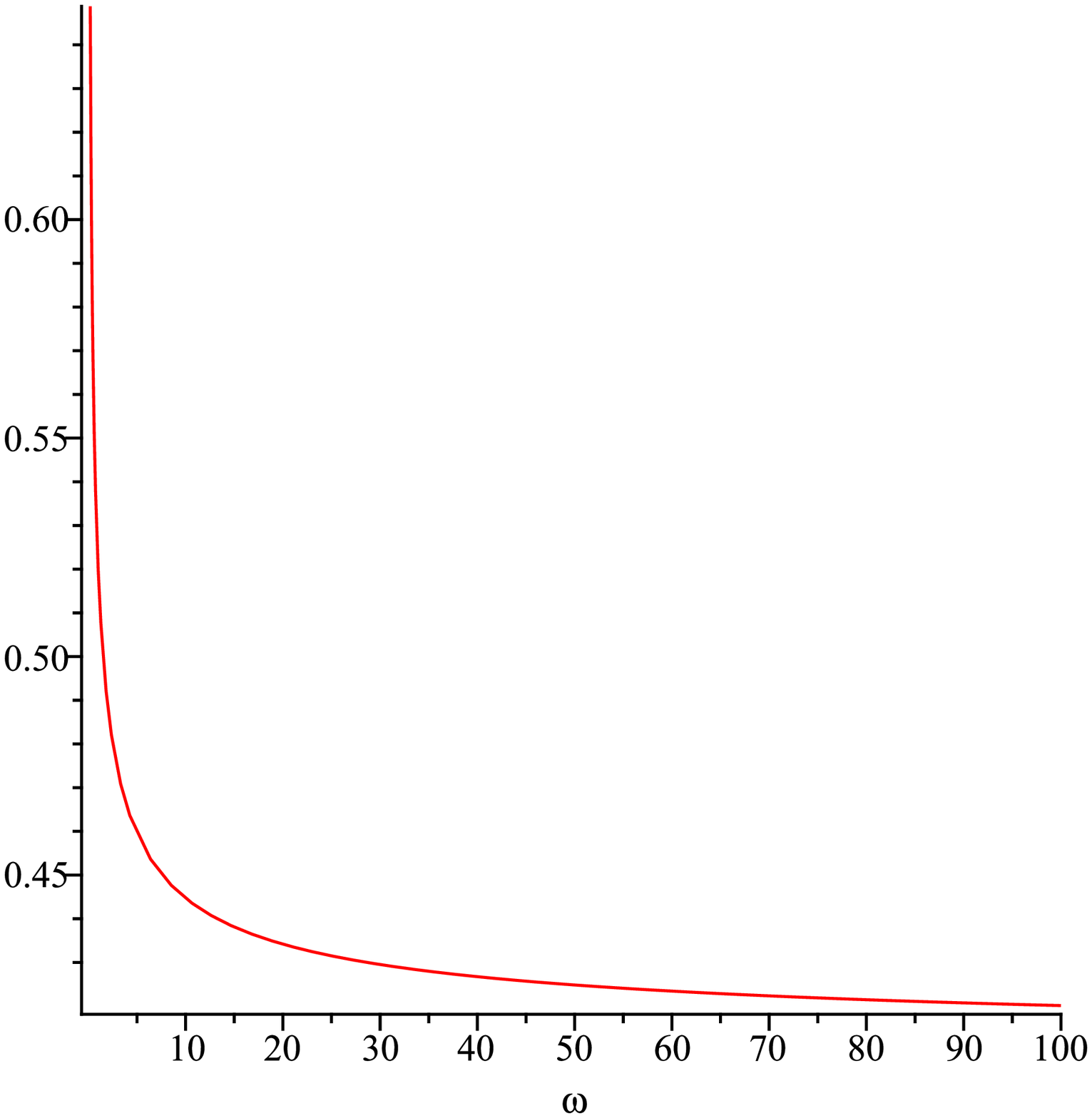}
\end{figure}

\indent Gathering the results above, we conclude that if $r\in\{1,2\}$ the solitary wave in $(\ref{solitary})$ is orbitally stable in $\Hh^1$. This fact can be precisely done by using Proposition $\ref{teo11}$, Theorem $\ref{specLline}$ and the behavior of $\mathcal{I}$ in terms of $r$ combined with the stability approaches in \cite{gss1}, \cite{gss2} and \cite{np}. For the case $r=4$, we need to use only \cite{gss1} applied to the space $\mathbb{H}_{e}^1$. Important to mention that in this last case, the Definition $\ref{stadef}$ must be considered only with the rotation symmetry since the translation symmetry is not invariant in the space $\mathbb{H}_{e}^1$. With this restriction one has $n(\mathcal{L})=1$ and $\ker(\mathcal{L})=[(0,\phi)]$ since the pair $(\phi',0)$ in the kernel of $\mathcal{L}$ defined in the whole space comes from the dropped symmetry. This fact finishes the proof of Theorem $\ref{maintheo}$ for the case $\mathbb{B}=\R$.\\

\begin{remark}
It is worth mentioning that in the general case, we are not able to determine any information about the orbital stability/instability. In fact, we have some difficulties to provide a satisfactory behavior of the quotient $\frac{a^2}{b}$ in terms of $\omega$ for these waves. Our intention is to give a positive answer for this question in a brief future.
\end{remark}
\noindent {\bf Case 2.} $\mathbb{B}=\mathbb{T}$.\\
\noindent $\bullet$ $r=1$. In this case, we can deduce from $(\ref{bk})$ that
\begin{equation}\label{normL2per}\displaystyle\int_{\mathbb{T}}\phi^2dx=\displaystyle\frac{2a^2E(k)}{b}=\frac{2\pi a^2E(k)}{K(k)},\end{equation} where $a$ is given by $(\ref{ak})$. In addition, from chain rule and the fact that $\frac{d\omega}{dk}>0$ (see Figure 8), we see that
$$\displaystyle\frac{d}{d\omega}\displaystyle\left(\displaystyle\int_{\mathbb{T}}\phi^2\ dx\right)=\displaystyle\frac{\displaystyle\frac{d}{dk}\int_{\mathbb{T}}\phi^2dx}{\frac{d\omega}{dk}}=\displaystyle\frac{\displaystyle\frac{d}{dk}\displaystyle\left(\frac{2\pi a^2 E(k)}{K(k)}\right)}{\frac{d\omega}{dk}}.$$

\indent Since $\frac{d\omega}{dk}>0$, we need to show that $\frac{d}{dk}\int_{\mathbb{T}}\phi^2dx>0$ for all $k\in (0,k_{*})$. Next picture shows us that $\int_{\mathbb{T}}\phi^2dx$ is strictly increasing for all these values:

\begin{figure}[!htb]
	\includegraphics[scale=0.35]{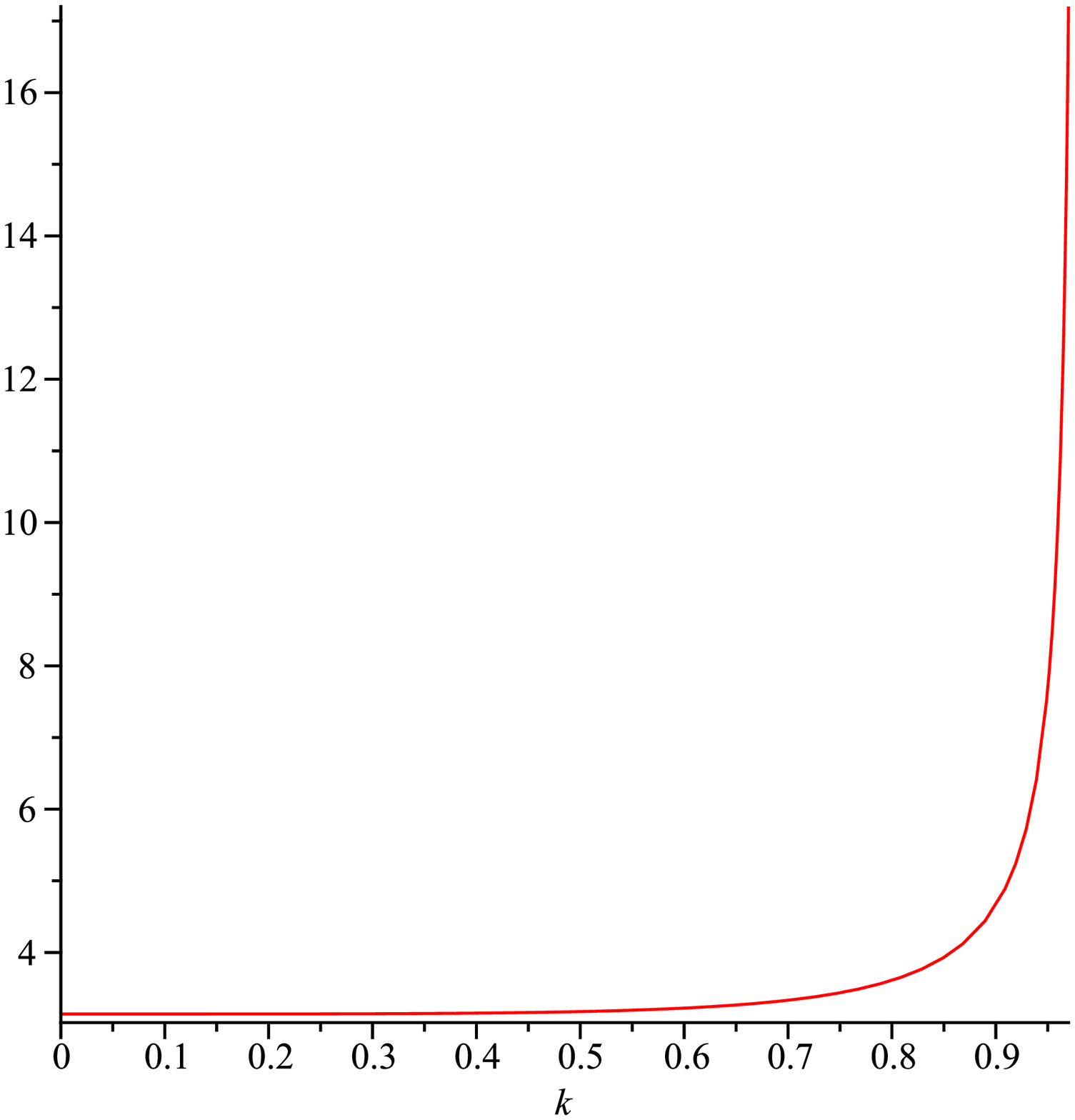}
\end{figure}

\noindent $\bullet$ $r=2$. In this situation, we can establish that
\begin{equation}\label{normL2per2}\displaystyle\int_{\mathbb{T}}\phi^2dx=\displaystyle\frac{2\pi a^2[k^2K(k)-k^2\Pi(\alpha,k)+\alpha\Pi(\alpha,k)]}{\alpha K(k)},\end{equation}
where $\Pi$ is the complete elliptic integral of the third kind (see \cite[Formula 110.08]{byrd}). As we have determined in the case $r=1$, we only need to prove that $\frac{d}{dk}\int_{\mathbb{T}}\phi^2dx>0$ (since Figure 10 gives us that $\frac{d\omega}{dk}>0$). We can plot the picture of $\int_{\mathbb{T}}\phi^2dx$ in terms of $k$ as:

\begin{figure}[!htb]
	\includegraphics[scale=0.35]{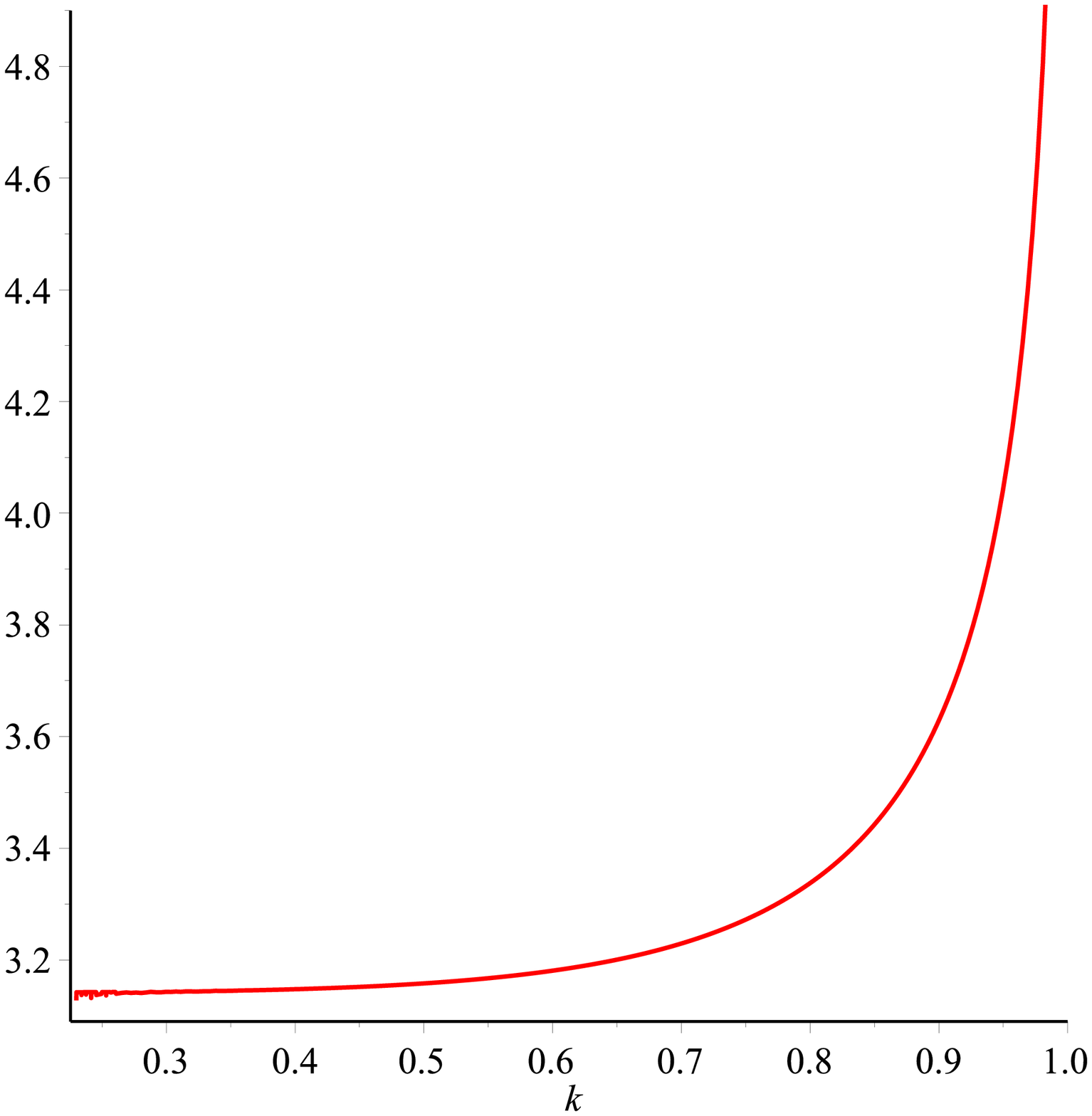}
\end{figure}

\indent Using the arguments in \cite{gss1}, \cite{gss2}, \cite{np}, we conclude that for the cases $r=1,2$, the periodic waves in $(\ref{per1})$ and $(\ref{per2})$ are orbitally stable in $\Hh^1$. This fact finishes the proof for the case $\mathbb{B}=\mathbb{T}$.
\section*{Acknowledgments}

F. Natali is partially supported by Funda\c c\~ao Arauc\'aria/Brazil (grant 002/2017) and CNPq/Brazil (grant 304240/2018-4).


\begin{thebibliography}{99}



\bibitem{an} J. Angulo, Nonlinear Dispersive Equations: Existence and Stability of Solitary and Periodic Travelling Waves Solutions, Mathematical Surveys and Monographs 156, Providence, 2009.

\bibitem{angulo1} J. Angulo, {\em Non-linear
	stability of periodic travelling-wave solutions for the
	Schr\"odinger and modified Korteweg-de Vries equation}, J.
Diff. Equat., 235 (2007), 1-30.

\bibitem{AN2}  J. Angulo and F. Natali, {\em Stability and instability of periodic
	travelling waves solutions for the critical Korteweg-de Vries and
	non-linear Schr\"odinger equations}, Physica D, 238 (2009),
603--621.

\bibitem{anto} P. Antonelli, A. Athanassoulis, H. Hajaiej and P. Markowich, {\em On the XFEL Schrödinger equation: highly oscillatory magnetic potentials and time averaging}, Arch. Ration. Mech. Anal., 211 (2014), 711–732.


\bibitem{byrd} P.F. Byrd and M.D. Friedman, Handbook of elliptic integrals
	for engineers and scientists, Springer, New York, 1971.




\bibitem{cazenave3} T. Cazenave, Semilinear Schr\"odinger
	equations, Courant Lect. Notes in Math., New York, 2003.






\bibitem{gss1} M. Grillakis, J. Shatah, and W. Strauss,
Stability theory of solitary waves in the presence of
symmetry I, {\em J. Funct. Anal.} 74 (1987), 160--197.

\bibitem{gss2} M. Grillakis, J. Shatah, and W. Strauss,
{Stability theory of solitary waves in the presence of
symmetry II}, {\em J. Funct. Anal.}, {74} (1990), 308--348.







\bibitem{magnus}W. Magnus and S. Winkler, Hill's equation, Wiley, New York, 1966.


\bibitem{nn} F. Natali and A. Neves, {\em Orbital stability of periodic waves}, IMA J. Appl. Math., 79 (2014), 1161-1179


\bibitem{np} F. Natali and A. Pastor, {\em The fourth-order dispersive nonlinear Schr\"odinger equation: orbital stability of a standing wave}, SIAM J. Appl. Dyn. Syst. 14 (2015), 1326--1347.





\bibitem{RS} M. Reed and B. Simon, Methods of Modern Mathematical Analysis: Analysis of Operators, Academic Press, Vol. IV, 1978.






\bibitem{we} M.I. Weinstein, Lyapunov stability of ground states of nonlinear dispersive evolutions equations, {\em Comm. Pure Appl. Math.} 39 (1986), 51--68.


\bibitem{zhang} J. Zhang, Z. Liu and M. Squassina, Modulational stability of ground states to nonlinear Kirchhoff equations, {\em J. Math. Anal. Appl.}, 447 (2019), 844--859.







\end{thebibliography}
\end{document}